\documentclass[a4paper,12pt]{article}

\usepackage[usenames,dvipsnames]{pstricks}

\usepackage{amscd,amssymb,amsmath,graphicx,verbatim,hyperref}
\usepackage[OT1,T1]{fontenc}

\usepackage[all, cmtip]{xy}
\usepackage{pdfpages}

\usepackage{amsmath,amsthm,amsfonts,mathabx,tikz,tikz-cd,stmaryrd,url,mathtools, mathrsfs, tensor}
\usepackage{hyperref}
\hypersetup{
	colorlinks,
	citecolor=black,
	filecolor=black,
	linkcolor=black,
	urlcolor=black
}

\usetikzlibrary{matrix,arrows,decorations.pathmorphing}

\newcounter{Definitioncount}
\newtheorem{theorem}{Theorem}

\newtheorem{proposition}[theorem]{Proposition}
\newtheorem{corollary}[theorem]{Corollary}

\theoremstyle{definition}
\newtheorem{definition}{Definition}
\newtheorem*{Definition}{Definition}
\newtheorem{remark}[theorem]{Remark}
\newtheorem{example}{Example}

\newtheoremstyle{fact}{\bigskipamount}{\medskipamount}{\upshape}{}{\itshape}{. }{ }{Fact}
\theoremstyle{fact}

\newtheoremstyle{quest}{\bigskipamount}{\medskipamount}{\upshape}{}{\itshape}{. }{ }{Question}
\theoremstyle{quest}

\newtheoremstyle{step}{2\bigskipamount}{\medskipamount}{\upshape}{}{\itshape}{. }{ }{\underline{Step~\thestep}}
\theoremstyle{step}

\renewcommand{\thestep}{\arabic{step}}

\newcommand{\lra}{\longrightarrow}
\newcommand{\Ra}{\Rightarrow}

\newcommand{\Lra}{\Longrightarrow}

\newcommand{\ldual}[1]{\mathord{{\let\nolimits\relax\sideset{^\wedge}{}{#1}}}}
\newcommand{\laction}[2]{\mathord{{\let\nolimits\relax\sideset{^{#1}}{}{#2}}}}
\newcommand{\conj}[2]{\mathord{{\let\nolimits\relax\sideset{^{#1}}{}{#2}}}}

\newcommand{\ox}{\otimes}
\newcommand{\xra}{\xrightarrow}
\newcommand{\xla}{\xleftarrow}
\newcommand{\xRa}{\xRightarrow}

\def\CA{{\mathscr A}}
\def\CB{{\mathscr B}}
\def\CC{{\mathscr C}}
\def\CD{{\mathscr D}}
\def\CE{{\mathscr E}}

\def\CK{{\mathscr K}}

\def\CM{{\mathscr M}}

\def\CP{{\mathscr P}}

\def\CT{{\mathscr T}}

\def\CV{{\mathscr V}}

\makeatletter
\newcommand*\bigcdot{\mathpalette\bigcdot@{.5}}
\newcommand*\bigcdot@[2]{\mathbin{\vcenter{\hbox{\scalebox{#2}{$\m@th#1\bullet$}}}}}
\makeatother

\newcommand{\twocong}[2][0.5]{\ar@{}[#2] \save ?(#1)*{\cong}\restore}
\newcommand{\twoeq}[2][0.5]{\ar@{}[#2] \save ?(#1)*{=}\restore}
\newcommand{\ltwocell}[3][0.5]{\ar@{}[#2] \ar@{=>}?(#1)+/r 0.2cm/;?(#1)+/l 0.2cm/^{#3}}
\newcommand{\rtwocell}[3][0.5]{\ar@{}[#2] \ar@{=>}?(#1)+/l 0.2cm/;?(#1)+/r 0.2cm/^{#3}}
\newcommand{\utwocell}[3][0.5]{\ar@{}[#2] \ar@{=>}?(#1)+/d  0.2cm/;?(#1)+/u 0.2cm/_{#3}}
\newcommand{\dtwocell}[3][0.5]{\ar@{}[#2] \ar@{=>}?(#1)+/u  0.2cm/;?(#1)+/d 0.2cm/^{#3}}
\newcommand{\ultwocell}[3][0.5]{\ar@{}[#2] \ar@{=>}?(#1)+/dr  0.2cm/;?(#1)+/ul 0.2cm/^{#3}}
\newcommand{\urtwocell}[3][0.5]{\ar@{}[#2] \ar@{=>}?(#1)+/dl  0.2cm/;?(#1)+/ur 0.2cm/^{#3}}
\newcommand{\dltwocell}[3][0.5]{\ar@{}[#2] \ar@{=>}?(#1)+/ur  0.2cm/;?(#1)+/dl 0.2cm/^{#3}}
\newcommand{\drtwocell}[3][0.5]{\ar@{}[#2] \ar@{=>}?(#1)+/ul  0.2cm/;?(#1)+/dr 0.2cm/^{#3}}

\begin{document}
\author{Ross Street \footnote{The author gratefully acknowledges the support of Australian Research Council Discovery Grant DP160101519.} \\ 
\small{Centre of Australian Category Theory} \\
\small{Macquarie University, NSW 2109 Australia} \\
\small{<ross.street@mq.edu.au>}
}
\title{Polynomials as spans}
\date{\small{\today}}
\maketitle

\noindent {\small{\emph{2010 Mathematics Subject Classification:} 18C15, 18D05}

\begin{abstract}
\noindent The paper defines polynomials in a bicategory $\mathscr{M}$. Polynomials in bicategories $\mathrm{Spn}\mathscr{C}$ of spans in a finitely complete category $\mathscr{C}$ agree with polynomials in $\mathscr{C}$ as defined by Nicola Gambino and Joachim Kock, and by Mark Weber. When $\mathscr{M}$ is \textit{calibrated}, we obtain another bicategory $\mathrm{Poly}\mathscr{M}$. We see that polynomials in $\mathscr{M}$ have representations as pseudofunctors $\mathscr{M}^{\mathrm{op}}\to \mathrm{Cat}$.
Using {\em tabulations}, we produce calibrations for the bicategory of relations in a regular category and for the bicategory of two-sided modules (distributors) between categories thereby providing new examples of bicategories of ``polynomials''.   
\end{abstract}

\tableofcontents

\section*{Introduction}

Polynomials in an internally complete (= ``locally cartesian closed'') category $\CE$
were shown by Gambino-Kock \cite{GambinoKock} to be the morphisms of a bicategory.
Weber \cite{Weber2015} defined polynomials in any category $\CC$ with pullbacks and proved they formed a bicategory. 
While both these papers are quite beautiful and accomplish further advances,
I felt the need to better understand the composition of polynomials. 
Perhaps what I have produced is merely a treatment of polynomials for bicategory theorists. 

The starting point was to view polynomials as spans of spans so that composition
could be viewed as the more familiar composition of spans using pullbacks; 
see B\'enabou \cite{Ben1967}. A polynomial from $X$ to $Y$ in a category $\CC$ 
is a diagram of the shape $X\xla{m_1} E\xra{m_2} S\xra{p} Y$ with $m_2$ 
a powerful (= exponentiable) morphisms in $\CC$. Such diagrams can be
thought of as generalizing spans: a span $X\xra{(m_1,S,p)}Y$ amounts to the
case where $E=S$ and $m_2$ is the identity. Our simple idea was to make the
diagram more complicated by including an identity thus:
\begin{eqnarray*}
X\xla{m_1} E\xra{m_2} S \xla{1_S} S \xra{p} Y \ ,
\end{eqnarray*}
resulting in a span 
\begin{eqnarray*}
X\xla{(m_1, E, m_2)} S \xra{(1_S, S, p)} Y
\end{eqnarray*}
of spans from $X$ to $Y$. 

Of course, the bicategory of spans does not have all bicategorical pullbacks.
Fortunately, polynomials are not general spans and sufficient pullbacks can be constructed.
Indeed, that is what Weber's distributivity pullbacks around a pair of composable
morphisms in $\CC$ construct. That construction requires the use of 
powerful morphisms in $\CC$. Here we define a morphism in
a bicategory to be a {\em right lifter} when every morphism into its codomain has
a right lifting through it. For spans in $\CC$ to be right lifters, one leg must be powerful. 

We introduce the term {\em calibration} for a class of morphisms, called {\em neat}, in a bicategory; the technical use of this word comes from B\'enabou \cite{Ben1975} who used it for categories.
A bicategory with a distinguished calibration is called {\em calibrated}. 
Polynomials in a calibrated bicategory $\CM$ are spans with one leg a right lifter and the other leg neat.
This suffices for the construction of a tricategory \cite{47} of polynomials in $\CM$ in which all the
3-morphisms are invertible. However, for two reasons, we decided to centre attention here on the 
bicategory $\mathrm{Poly}\CM$ obtained by taking isomorphism classes of 2-morphisms.
One reason is that it covers our present examples, the other is the possibility of iterating the
construction without moving to higher level categories. 

A {\em polynomic bicategory} $\CM$ is one in which the neat morphisms are all the groupoid fibrations (see Section~\ref{dfib}) in $\CM$. 
We prove that $\mathrm{Spn}\CC$ is polynomic for any finitely complete $\CC$.
In this case the polynomials are the polynomials in $\CC$ in the sense of Weber \cite{Weber2015}. 

The bicategory $\mathrm{Rel}\CE$ of relations in a regular category $\CE$ is calibrated by morphisms which are isomorphic to graphs of monomorphisms in $\CE$. 
In Example~\ref{toposex} for $\CE$ a topos, we give a reinterpretation of the bicategory of polynomials in 
$\mathrm{Rel}\CE$ as a Kleisli construction. 

By providing a calibration for the bicategory $\mathrm{Mod}$ of two-sided modules between categories, we obtain another example. Again, in Example~\ref{Modex} , we give a reinterpretation of the bicategory of polynomials in $\mathrm{Mod}$ as a Kleisli construction. 

It must be pointed out that the meaning of polynomial in a bicategory is different from the meaning in Section 4 of Weber \cite{Weber2015} which is about polynomials in 2-categories. 
Weber is dealing with the 2-category as a $\mathrm{Cat}$-enriched category, taking the polynomials to be diagrams of the same shape as in the case of ordinary categories, and accommodating the presence of 2-cells. 
In particular, if a category is regarded as a 2-category with only identity 2-cells,
then his polynomials in the 2-category are just polynomials in the category.
To define a polynomial, in the sense of this paper, in such a 2-category would require the specification of a calibration on the category and then a polynomial would reduce to a single morphism (called ``neat'') in that calibration.   
 
I am grateful to the Australian Category Seminar, especially Yuki Maehara,
Richard Garner, Michael Batanin and Charles Walker, for comments during and following my talks on this topic.
I am also particularly grateful to the diligent and insightful referee for suggesting important improvements,
mainly that I should add the detail to the previously vaguely expressed 
Examples~\ref{toposex} and \ref{Modex}; there are categorical facts involved that may not be so well known.

\section{Bipullbacks and cotensors}

Recall that the pseudopullback (also called iso-comma category) of two functors $\CC \xra{F}\CE \xla{P} \CD$ is the category $F/_{\mathrm{ps}}P$ whose objects $(C,\alpha, D)$ consist of objects $C\in \CC$ and $D\in \CD$ with $\alpha : FC\xra{\cong}PD$,
and whose morphisms $(u,v) : (C,\alpha, D)\to (C',\alpha', D')$ consist of morphisms
$u : C\to C'$ in $\CC$ and $v : D\to D'$ in $\CD$ such that $(Pv)\alpha = \alpha'(Fu)$. We have a universal square of functors
 \begin{equation}\label{pspb}
 \begin{aligned}
\xymatrix{
F/_{\mathrm{ps}}P \ar[d]_{\mathrm{cod}}^(0.5){\phantom{AAAA}}="1" \ar[rr]^{\mathrm{dom}}  && \CD \ar[d]^{F}_(0.5){\phantom{AAAA}}="2" \ar@{<=}"1";"2"^-{\xi}_-{\cong}
\\
\CC \ar[rr]_-{P} && \CE 
}
 \end{aligned}
\end{equation} 
containing an invertible natural transformation $\xi$. 

A square
 \begin{equation}\label{bipb}
 \begin{aligned}
\xymatrix{
P \ar[d]_{c}^(0.5){\phantom{aaaa}}="1" \ar[rr]^{d}  && A \ar[d]^{n}_(0.5){\phantom{aaaa}}="2" \ar@{<=}"1";"2"^-{\theta}_-{\cong}
\\
B \ar[rr]_-{p} && C 
}
 \end{aligned}
\end{equation} 
in a bicategory $\CA$ is a {\em bipullback} of the cospan $A\xra{n}C\xra{p}B$
when, for all objects $K$ of $\CA$, the induced functor
$$\CA(K,P)\lra \CA(K,n)/_{\mathrm{ps}}\CA(K,p) \ , \  u \mapsto (du,\theta u, cu) \ ,$$ 
is an equivalence of categories.

In a bicategory $\CA$, we write $A^{\mathbf{2}}$ for the (bicategorical) cotensor
(or power) of $A$ with the ordinal $\mathbf{2}$; this means that the category $\CA(K,A^{\mathbf{2}})$ is equivalent to the arrow category of $\CA(K,A)$,
pseudonaturally in $K\in \CA$. The identity morphism in 
$\CA(A^{\mathbf{2}},A^{\mathbf{2}})$ corresponds to a morphism (arrow)
$$
 \xymatrix{
 A^{\mathbf{2}}   \ar@/_1pc/[rr]_{\mathrm{c}}  \ar@/^1pc/[rr]^{\mathrm{d}} & \Downarrow_{^\lambda} & A 
   } 
 $$ 
 in $\CA(A^{\mathbf{2}},A)$.   
 
 \begin{example}
 For $\CA = \CV\text{-}\mathrm{Cat}$ in the sense of \cite{KellyBook}, the $\CV$-category $A^{\mathbf{2}}$ is the usual arrow $\CV$-category.  
 \end{example}
 
 \begin{example}\label{powersinMod}
 Recall (from \cite{22} for instance) the definition of the bicategory $\CV\text{-}\mathrm{Mod}$ of $\CV$-categories and their modules for a nice symmetric closed monoidal base category $\CV$. The objects are $\CV$-categories. The homcategories are defined to be the
 $\CV$-functor categories 
 \begin{eqnarray*}
\CV\text{-}\mathrm{Mod}(A,B)= [B^{\mathrm{op}}\ox A, \CV]
\end{eqnarray*}
whose objects $m : B^{\mathrm{op}}\ox A\to \CV$ are called modules from $A$ to $B$. Composition is defined by the coends
$(n\circ m)(c,a)=\int^b{m(b,a)\ox n(c,b)}$.
Let $I_*\mathbf{2}$ denote the free $\CV$-category on the category $\mathbf{2}$. 
The cotensor of the $\CV$-category $A$ with the ordinal $\mathbf{2}$ in the bicategory $\CV\text{-}\mathrm{Mod}$ is the $\CV$-category $(I_*\mathbf{2})^{\mathrm{op}}\ox A$.
This is because of the calculation  
 \begin{align*}
\CV\text{-}\mathrm{Mod}(K,(I_*\mathbf{2})^{\mathrm{op}}\ox A) & = [I_*\mathbf{2}\ox A^{\mathrm{op}}\ox K, \CV] \\
& \cong [A^{\mathrm{op}}\ox K, \CV^\mathbf{2}] \\ 
& \cong [A^{\mathrm{op}}\ox K, \CV]^{\mathbf{2}} \\
& \cong \CV\text{-}\mathrm{Mod}(K,A)^{\mathbf{2}} \ .
\end{align*}
Let $\partial_0 :\mathbf{1}\to \mathbf{2}$ be the functor $0\mapsto 1$; 
it is right adjoint to $ ! : \mathbf{2}\to \mathbf{1}$. It follows that $\mathrm{c} : A^{\mathbf{2}} \to A$ in $\CV\text{-}\mathrm{Mod}$
is $((I_*!)^{\mathrm{op}}\ox A)_* : (I_*\mathbf{2})^{\mathrm{op}}\ox A \to A$.
In particular, when $\CV = \mathrm{Set}$, $\mathrm{c}$ is the module $\mathrm{pr}_{2 \ *}$ induced by
the second projection functor $\mathbf{2}^{\mathrm{op}}\times A \to A$.     
  \end{example}

\begin{remark}\label{lowerstaronbicolimits}
The phenomenon described in Example~\ref{powersinMod} has to do with the fact that the pseudofunctor $(-)_* : \CV\text{-}\mathrm{Cat}\to \CV\text{-}\mathrm{Mod}$,
taking each $\CV$-functor $f : A\to B$ to the module $f_*:A\to B$ with $f_*(b,a) = B(b,fa)$, preserves bicolimits and $\CV\text{-}\mathrm{Mod}$ is self dual.
\end{remark}

\section{Groupoid fibrations}\label{dfib}

Let $p : E \to B$ be a functor. A morphism $\chi : e' \to e$ in $E$ is called {\em cartesian}\footnote{Classically called ``strong cartesian''} for $p$ when the square \eqref{cart} is a pullback for all $k\in E$.
\begin{equation}\label{cart}
\begin{aligned}
\xymatrix{
E(k,e') \ar[rr]^-{E(k,\chi)} \ar[d]_-{p} && E(k,e) \ar[d]^-{p} \\
B(pk,pe') \ar[rr]_-{B(pk,p\chi)} && B(pk,pe)}
\end{aligned}
\end{equation}
Note that all invertible morphisms in $E$ are cartesian.
If $p$ is fully faithful then all morphisms of $E$ are cartesian.

We call the functor $p : E \to B$ a {\em groupoid fibration} when 
\begin{itemize}
\item[(i)] for all objects $e\in E$
and morphisms $\beta : b \to pe$ in $B$, there exist a morphism $\chi : e'\to e$ in $E$
and isomorphism $b\cong pe'$ whose composite with $p\chi$ is $\beta$, and
\item[(ii)] every morphism of $E$ is cartesian for $p$.  
\end{itemize}
From the pullback \eqref{cart}, it follows that groupoid fibrations are conservative (that is, reflect invertibility).  

We call the functor $p : E \to B$ an {\em equivalence relation fibration} or {\em er-fibration} when it is a groupoid fibration and the only endomorphisms $\xi : x \to x$ in $E$ which map to identities under $p$ are identities.
It follows (using condition (ii)) that $p$ is faithful.
Note that if $p$ is an equivalence then it is an er-fibration.

Write $\mathrm{GFib}B$ for the 2-category whose objects are groupoid fibrations $p : E\to B$, and whose hom categories are given by the following pseudopullbacks.
 \begin{equation}\label{GFibhom}
 \begin{aligned}
\xymatrix{
\mathrm{GFib}B(p,q) \ar[d]_{}^(0.5){\phantom{AAAAAA}}="1" \ar[rr]^{}  && [E,F] \ar[d]^{[E,q]}_(0.5){\phantom{AAAAAA}}="2" \ar@{<=}"1";"2"^-{\cong}_-{}
\\
\mathbf{1} \ar[rr]_-{\lceil p\rceil} && [E,B] 
}
 \end{aligned}
\end{equation}
So objects of $\mathrm{GFib}B(p,q)$ are pairs $(f,\phi)$ where $f : E \to F$
is a functor and $\phi : q f \Ra p$ is an invertible natural transformation.
If $\phi$ is an identity then $(f,\phi)$ is called {\em strict}.  

Let $\mathrm{Gpd}$ be the 2-category of groupoids, functors and natural transformations. 
Write $\mathrm{Hom}(B^{\mathrm{op}},\mathrm{Gpd})$ for the 2-category of pseudofunctors (= homomorphisms of bicategories \cite{Ben1967}) $T :  B^{\mathrm{op}}\to \mathrm{Gpd}$, pseudo-natural transformations, and modifications \cite{7}.

Recall that the Grothendieck construction $\mathrm{pr} : \wr T \to B$
on a pseudofunctor $T :  B^{\mathrm{op}}\to \mathrm{Gpd}$ is the projection functor from
the category $\wr T$ whose objects are pairs $(t,b)$ with $b\in B$
and $t\in Tb$, and whose morphisms $(\tau,\beta) : (t,b)\to (t',b')$ consist of
morphisms $\beta : b\to b'$ in $B$ and $\tau : t\to (T\beta)t'$ in $Tb$. 
This construction is the object assignment for a 2-functor
\begin{eqnarray}\label{El}
\wr  :  \mathrm{Hom}(B^{\mathrm{op}},\mathrm{Gpd}) \lra \mathrm{GFib}B
\end{eqnarray}
which actually lands in the sub-2-category of strict morphisms.
Note that the pullback
\begin{equation*}
\xymatrix{
Tb \ar[rr]^-{} \ar[d]_-{} && \wr T \ar[d]^-{\mathrm{pr}} \\
\mathbf{1} \ar[rr]_-{\lceil b \rceil} && B}
\end{equation*}
is also a bipullback (see \cite{42}); this suggests that we can reconstruct
a pseudofunctor $T$ from a groupoid fibration $p : E\to B$ by defining
$Tb$ to be the pseudopullback of $\lceil b \rceil : \mathbf{1}\to B$ and $p$.
\begin{proposition}
The 2-functor \eqref{El} is a biequivalence.
\end{proposition}

A category which is both a groupoid and a preorder is the same as an equivalence relation; that is, a set of objects equipped with an equivalence relation thereon.
Let $\mathrm{ER}$ be the 2-category of equivalence relations, functors and natural transformations. 
Note that the 2-functor $\mathrm{Set} \to \mathrm{ER}$ taking each set to the identity relation is a biequivalence. 
Write $\mathrm{ERFib}B$ for the full sub-2-category of $\mathrm{GFib}B$
with objects the er-fibrations. 
\begin{proposition}\label{Eler}
The biequivalence \eqref{El} restricts to a biequivalence
\begin{eqnarray*}
\wr  :  \mathrm{Hom}(B^{\mathrm{op}},\mathrm{ER}) \xra{\sim} \mathrm{ERFib}B \ ,
\end{eqnarray*}
and so further restricts to a biequivalence
\begin{eqnarray*}
[B^{\mathrm{op}},\mathrm{Set}] \xra{\sim} \mathrm{ERFib}B \ .
\end{eqnarray*}
\end{proposition}

Let $\CE$ and $\CB$ be bicategories. Baklovi\'c \cite{Bak2011} and Buckley \cite{Buckley2014} say that a morphism $x : Z\to X$ in $\CE$ is {\em cartesian} for a pseudofunctor $P : \CE \to \CB$ when the following square is a bipullback in $\mathrm{Cat}$ for all objects $K$ of $\CE$.
\begin{equation}\label{cartmor}
 \begin{aligned}
\xymatrix{
\CE(K,Z) \ar[d]_{P}^(0.5){\phantom{aaaaaaaa}}="1" \ar[rr]^{\CE(K,x)}  && \CE(K,X) \ar[d]^{P}_(0.5){\phantom{aaaaaaaa}}="2" \ar@{<=}"1";"2"^-{\cong}_-{}
\\
\CB(PK,PZ) \ar[rr]_-{\CB(PK,Px)} && \CB(PK,PX) 
}
 \end{aligned}
\end{equation} 
A 2-cell $\sigma : x' \Ra x : Z \to X$ in $\CE$ is called {\em cartesian} for $P$
when it is cartesian (as a morphism of $\CE(Z,X)$) for the functor $P : \CE(Z,X) \to \CB(PZ,PX)$.  
Note that all equivalences are cartesian morphisms and all invertible 2-cells are cartesian.
\begin{definition}\label{defgfib}
A pseudofunctor $P : \CE \to \CB$ is a {\em groupoid fibration} when 
\begin{itemize}
\item[(i)] for all $X\in \CE$ and $f : B\to PX$ in $\CB$, there exist a morphism $x : Z\to X$ in $\CE$ and an equivalence $B\simeq PZ$ whose composite with $Px$ is isomorphic to $f$,
\item[(ii)] every morphism of $\CE$ is cartesian for $P$, and 
\item[(iii)] every 2-cell of $\CE$ is cartesian for $P$. 
\end{itemize}
A morphism $p : E \to B$ in a tricategory $\CT$ is called a {\em groupoid fibration}
when, for all objects $K$ of $\CT$, the pseudofunctor $\CT(K,p) : \CT(K,E)\to \CT(K,B)$ is a groupoid fibration between bicategories.
\end{definition}
\begin{definition}\label{defconserv}
A pseudofunctor $F : \CA \to \CB$ is called {\em conservative} when
\begin{itemize}
\item[(a)] if $Ff$ is an equivalence in $\CB$ for a morphism $f$ in $\CA$
then $f$ is an equivalence; 
\item[(b)] if $F\alpha$ is an isomorphism in $\CB$ for a 2-cell $\alpha$ in $\CA$
then $\alpha$ is an isomorphism.
\end{itemize}
A morphism $f : A \to B$ in a tricategory $\CT$ is {\em conservative}
when, for all objects $K$ of $\CT$, the pseudofunctor $\CT(K,f) : \CT(K,A)\to \CT(K,B)$ is conservative.   
\end{definition}
\begin{proposition}\label{gfibsconserve}
Groupoid fibrations are conservative.
\end{proposition}
\begin{proof}
If $Px$ is an equivalence, we see from the bipullback \eqref{cartmor} that each functor $\CE(K,x)$ is too. Since these equivalences can be chosen to be adjoint equivalences, they become pseudonatural in $K$ and so, by the bicategorical Yoneda Lemma \cite{14}, are represented by an inverse equivalence for $x$. This proves (a) in the Definition of conservative. Similarly, for (b), look at the pullback \eqref{cart} for the functor $p = (\CE(Z,X) \xra{P} \CB(PZ,PX))$. 
\end{proof}

 For a pseudofunctor $P : \CE \to \CB$ between bicategories, write $\CB/P$ for the bicategory whose objects are pairs $(B\xra{f}PE,E)$, where $E$ is an object of $\CE$ and $f:B\to PE$ is a morphism of $\CB$, and whose homcategories are defined by pseudopullbacks
 \begin{equation}\label{B/P}
 \begin{aligned}
   \xymatrix{\CB/P((f,E),(f',E')) \ar[dd]_{\mathrm{d}} \ar[rr]^{\mathrm{c}} & & \CE(E,E') \ar[d]^{P} \\
      &   &  \dtwocell{ll}{\cong} \CB(PE,PE') \ar[d]^{\CB(f,1)} \\
     \CB(B,B')\ar[rr]_{\CB(1,f')} & & \CB(B,PE')}
 \end{aligned}
\end{equation}
Write $\CE/\CE$ for $\CB/P$ in the case $P$ is the identity pseudofunctor of $\CE$.
There is a canonical pseudofunctor $J_P : \CE/\CE \to \CB/P$ taking 
the object $(X\xra{u}E,E)$ to $(PX\xra{Pu}PE,E)$. 

\begin{proposition}\label{propgfib}
The pseudofunctor $P : \CE \to \CB$ between bicategories satisfies condition (i) in the Definition~\ref{defgfib} of groupoid fibration if and only if $J_P : \CE/\CE \to \CB/P$ is surjective on objects up to equivalence. 
Also, $P : \CE \to \CB$ satisfies condition (ii) if and only if the effect of $J_P : \CE/\CE \to \CB/P$ on homcategories is an equivalence. 
Condition (iii) is automatic if all 2-cells in $\CE$ are invertible. 
\end{proposition}

\begin{example}
Each biequivalence of bicategories is a groupoid fibration.
\end{example}
\begin{example}\label{domgfib}
Let $H$ be an object of the bicategory $\CB$.
Write $\CB/H$ for the bicategory $\CB/P$ where $P$ is the constant pseudofunctor
$\mathbf{1}\to \CB$ at $H$. The ``take the domain'' pseudofunctor 
\begin{eqnarray}\label{dom}
\mathrm{dom} :\CB/H \to \CB
\end{eqnarray}
is a groupoid fibration. 
For, it is straightforward to see that the canonical pseudo-functor 
$(\CB/H)/(\CB/H)\to \CB/\mathrm{dom}$ is a biequivalence,
so it remains to prove each 2-cell 
$$\sigma : (g,\psi) \Ra (f,\phi) : (A\xra{a}H)\to (B\xra{b}H)$$
in $\CB/H$ is cartesian for \eqref{flowerstar}.
The condition for a 2-cell is $(b\sigma)\psi = \phi$.
Take another 2-cell $\tau : (h,\theta) \Ra (f,\phi)$ in $\CB/H$
(so that $(b\tau)\theta = \phi$) and a 2-cell $\upsilon : h\Ra g$
in $\CB$ such that $\sigma \upsilon = \tau$. 
Then we have $(b\sigma)(b\upsilon)\theta = (b\sigma\upsilon)\theta = (b\tau)\theta 
= \phi = (b\sigma)\psi$ with $b\sigma$ invertible. So $(b\upsilon)\theta = \psi$.
We conclude that $\upsilon : (h,\theta) \Ra (g,\psi)$ is a 2-cell in $\CB/H$, as required.

Note that \eqref{dom} is not a local groupoid fibration in general; 
that is, the functor $\mathrm{dom}_{a,b} : \CB/H(A\xra{a}H,B\xra{b}H) \to \CB(A,B)$ is generally not a groupoid fibration.   
\end{example}
\begin{example}\label{amg}
Apparently more generally, let $f : H\to K$ be a morphism in the bicategory $\CB$.
Write 
\begin{eqnarray}\label{flowerstar}
f_* : \CB/H \to \CB/K
\end{eqnarray}
for the pseudofunctor which composes with $f$.
On applying Example~\ref{domgfib} with $\CB$ and $H$ replaced by $\CB/K$ and $H\xra{f}K$, up to biequivalence we obtain this example.     
\end{example}

\begin{proposition}\label{gfibterm}
Up to biequivalence, pseudofunctors of the form \eqref{dom} are precisely the groupoid fibrations $P : \CE \to \CB$ for which the domain bicategory has a terminal object. Moreover, if such a $P$ has a left biadjoint, it is a biequivalence.   
\end{proposition}
\begin{proof}
Let $T$ be a terminal object of $\CE$. 
Make a choice of morphism $!_E : E \to T$ for each object $E$ of $\CE$. 
Then the pseudofunctor 
$$\hat{J}_P : \CE \lra \CB/PT \ , \ E \mapsto (PE\xra{P!_E}PT)$$
is a biequivalence over $\CB$; it is a tripullback of the biequivalence $J_P$ 
along the canonical $\CB/PT \to \CB/P$. 
So $P$ is biequivalent to \eqref{dom} with $H=PT$. 

For the second sentence, if we suppose the $\mathrm{dom}$ of \eqref{dom} has a left biadjoint then it preserves terminal objects. The bicategory $\CB/H$ has the terminal object $1_H : H \to H$. So $H = \mathrm{dom}(H\xra{1_H}H)$ is terminal in $\CB$.
So $\mathrm{dom}$ is a biequivalence.    
\end{proof}

\section{Spans in a bicategory}\label{siabc}

Spans in a bicategory $\CA$ with bipullbacks (= iso-comma objects) will be recalled;
compare \cite{14} Section 3.  

A {\em span} from $X$ to $Y$  in the bicategory $\CA$ is a diagram $X\xla{u}S\xra{p}Y$; we write $(u,S,p) : X\to Y$.
 The composite of $X\xla{u}S\xra{p}Y$ and $Y\xla{v}T\xra{q}Z$ is obtained from the diagram
 \begin{eqnarray}\label{compspans}
 \begin{aligned}
\xymatrix{
& & P \ar[ld]_-{\mathrm{pr}_1} \ar[rd]^-{\mathrm{pr}_2} & & \\
& S\ar @{} [rr] | {\stackrel{\cong} \Longleftarrow} \ar[ld]_-{u} \ar[rd]^-{p} & & T \ar[ld]_-{v} \ar[rd]^-{q} & \\
X & & Y & & Z}
 \end{aligned}
\end{eqnarray}
 (where $P$ is the bipullback of $p$ and $v$) as the span $X\xla{u\mathrm{pr}_1}P\xra{q\mathrm{pr}_2}Z$.
 A {\em morphism $(\lambda, h, \rho) : (u,S,p)\to (u',S',p')$ of spans} is a morphism $h : S\to S'$ in $\CM$
equipped with invertible 2-cells as shown in the two triangles below.
  \begin{eqnarray}\label{spnmorph}
\begin{aligned}
\xymatrix{
& & S \ar[d]|-{ h} \ar[lld]_-{u} \ar[rrd]^-{ p} \ar @{} [ld] | {\stackrel{\lambda\cong} \Leftarrow} 
\ar @{} [rd] | {\stackrel{\rho\cong} \Leftarrow} & &
\\
X   && S' \ar[ll]^-{u'} \ar[rr]_-{p'} & & Y  }
\end{aligned}
\end{eqnarray}
A 2-cell $\sigma : h\Ra k : (u,S,p)\to (u',S',p')$ between such morphisms is a 2-cell 
$\sigma : h \Ra k : S \to S'$ in $\CM$ which is compatible with the 2-cells
in the triangles in the sense that $\lambda = \lambda'.u'\sigma$ and $\rho'=p'\sigma.\rho$.
We write $\mathrm{Spn}\CA(X,Y)$ for the bicategory of spans from $X$ to $Y$.  
 
 Composition pseudofunctors
 $$\mathrm{Spn}\CA(Y,Z)\times \mathrm{Spn}\CA(X,Y)\lra \mathrm{Spn}\CA(X,Z)$$
 are defined on objects by composition of spans \eqref{compspans} and on morphisms by using the universal properties of bipullback.
 
 In this way, we obtain a tricategory $\mathrm{Spn}\CA$. The associator equivalences are obtained using the horizontal and vertical stacking properties of
 pseudopullbacks. The identity span on $X$ has the form $(1_X,X,1_X)$ and
 the unitor equivalences are obtained using the fact that a pseudopullback of
 the cospan $X\xra{f}Y\xla{1_Y}Y$ is given by the span $X\xla{1_X}X\xra{f}Y$ equipped with the canonical isomorphism $1_Yf\cong f \cong f1_X$ in $\CA$.
  
For $e : X\to Y$ in $\CA$, let $e_*=(1_X,X,f) : X\to Y$.
Notice that $e^*=(e,X,1_X) : Y\to X$ is a right biadjoint for $e_*$ in the tricategory $\mathrm{Spn}\CA$.        

\begin{proposition}\label{charleftadj}
Let $\CA$ be a finitely complete bicategory.
The following conditions on a span $(u,S,p)$ from $X$ to $Y$ in $\CA$ are equivalent:
\begin{itemize}
\item[(i)] the morphism $(u,S,p): X\to Y$ has a right biadjoint in the tricategory $\mathrm{Spn}\CA$;
\item[(ii)] the morphism $u : S \to X$ is an equivalence in $\CA$;
\item[(iii)] the morphism $(u,S,p): X\to Y$ is equivalent in $\mathrm{Spn}\CA$ to $f_*$ for some morphism $f$ in $\CA$;
\item[(iv)] the morphism $(u,S,p): X\to Y$ is a groupoid fibration in the tricategory $\mathrm{Spn}\CA$.
\end{itemize}
\end{proposition}
\begin{proof} 
The equivalence of (i), (ii) and (iii) is essentially as in the case where $\CA$ is a category; see \cite{26}. We will prove the equivalence of (ii) and (iv).
Put $\mathfrak{S}=\mathrm{Spn}\CA$ to save space.  
Using Proposition~\ref{gfibterm} and the fact that the span 
$K\xla{\mathrm{pr_1}}K\times X\xra{\mathrm{pr_2}}X$ is a terminal object in
the bicategory $\mathfrak{S}(K,X)$, we see that the pseudofunctor $P_K : =  \mathfrak{S}(K,X)\xra{\mathfrak{S}(K,p_*u^*)} \mathfrak{S}(K,Y)$ is
a groupoid fibration if and only if the canonical pseudofunctor 
$J_{P_K}$ in the diagram \eqref{J_P_K}
is a biequivalence.
\begin{eqnarray}\label{J_P_K}
\begin{aligned}
\xymatrix{
& \mathfrak{S}(K,X) \ar[ld]_{\mathfrak{S}(K,u^*)}^(0.5){\phantom{AAAAAA}}="1"   \ar[rd]^{J_{P_K}}_(0.5){\phantom{AAAAAA}}="2" \ar@{=>}"1";"2"^-{\simeq}
\\
\mathfrak{S}(K,S) \ar[rr]_-{J_{\mathfrak{S}(K,p_*)}} && \mathfrak{S}(K,Y)/(\mathrm{pr_1},K\times S, p\mathrm{pr_2}) 
}
\end{aligned}
\end{eqnarray}
However, we see easily that $J_{P_K}$ does factor up to equivalence as shown in \eqref{J_P_K} where $J_{\mathfrak{S}(K,p_*)}$ is a biequivalence. 
So $p_*u^* :X\to Y$ is a groupoid fibration if and only if $\mathfrak{S}(K,u^*)$ is
a biequivalence for all $K$; that is, if and only $u$ is an equivalence in $\CA$.     
\end{proof}

\begin{remark}
\begin{itemize}
\item[(a)] In fact (ii) implies (iv) in Proposition~\ref{charleftadj} requires no assumption on
the bicategory $\CA$. For, it is straightforward to check (compare Example~\ref{amg}) that $p_* : \mathrm{Spn}\CA(K,S)\to \mathrm{Spn}\CA(K,Y)$ is a groupoid fibration for all $K$; 
this does not even require bipullback in $\CA$ since we only need the hom bicategories of $\mathrm{Spn}\CA$. 
\item[(b)] Given that $p_*$ is a groupoid fibration, we can prove the converse (iv) implies (ii) by noting that $p_*u^*$ is a groupoid fibration if and only if $u^*$ is 
(compare (i) of Proposition~\ref{bipbproperties}). So, provided $\mathrm{Spn}\CA(K,Y)$ has a terminal
object (as guaranteed by the finite bicategorical limits in $\CA$), we deduce
that $u^*$ is a biequivalence using Proposition~\ref{gfibterm} and $u_*\dashv u^*$.   
\end{itemize}  
\end{remark}

\begin{remark}
If $\CC$ is a finitely complete category (regarded as a bicategory with only identity 2-cells) then the tricategory $\mathrm{Spn}\CC$ has only identity 3-cells; it is a bicategory. 
We are interested in spans in such a bicategory $\CA = \mathrm{Spn}\CC$.
The problem is that bipullbacks do not exist in this kind of $\CA$ in general.
Hence we must hone our concepts to restricted kinds of spans in $\CA$.
\end{remark}

\section{More on bipullbacks and groupoid fibrations}\label{bgf}

In Section~\ref{siabc}, we defined groupoid fibrations in a tricategory. This applies in a bicategory $\CA$ regarded as a tricategory by taking only identity 3-cells.
Then the 2-cells in each $\CA(A,B)$ are invertible (identities in fact) so condition (iii) of Definition~\ref{dfib} is automatic.

\begin{proposition}\label{dfib_cotensor} 
Suppose $p : E\to B$ is a morphism in a bicategory $\CA$ for which  $E^{\mathbf{2}}$ and $B^{\mathbf{2}}$ exist.
Then $p : E\to B$ is a groupoid fibration in $\CA$ if and only if the square
 \begin{equation*}
 \begin{aligned}
\xymatrix{
E^{\mathbf{2}} \ar[d]_{p^{\mathbf{2}}}^(0.5){\phantom{aaaa}}="1" \ar[rr]^{\mathrm{c}}  && E \ar[d]^{p}_(0.5){\phantom{aaaa}}="2" \ar@{<=}"1";"2"^-{\cong}
\\
B^{\mathbf{2}} \ar[rr]_-{\mathrm{c}} && B 
}
 \end{aligned}
\end{equation*}
is a bipullback.  
\end{proposition}
\begin{proof}
Since all concepts are defined representably, it suffices to check this for the bicategory $\CA = \mathrm{Cat}$ where the bipullback of $\mathrm{c}$ and $p$
is the comma category $B/p$. So the square in the Proposition is a bipullback if and only if the canonical functor $j_p : E^{\mathbf{2}} \to B/p$ is an equivalence. We have the result by looking at Proposition~\ref{propgfib} in the bicategory case.   
\end{proof}

\begin{proposition}\label{bipbproperties} 
Suppose $\CA$ is a bicategory.
\begin{itemize}
\item[(i)] Suppose $r \cong p q$ with $p$ a groupoid fibration in $\CA$. Then $r$ is a groupoid fibration if and only if $q$ is.
\item[(ii)] In the bipullback \eqref{bipb} in $\CA$, if $p$ is a groupoid fibration then so is $d$.  
\item[(iii)] Suppose \eqref{bipb} is a bipullback in $\CA$ and $p$ is a groupoid fibration. For any square
 \begin{equation}\label{laxsquare}
 \begin{aligned}
\xymatrix{
K \ar[d]_{v}^(0.5){\phantom{aaaa}}="1" \ar[rr]^{u}  && A \ar[d]^{n}_(0.5){\phantom{aaaa}}="2" \ar@{<=}"1";"2"^-{\psi}
\\
B \ar[rr]_-{p} && C 
}
 \end{aligned}
\end{equation} 
in $\CA$ with $\psi$ not necessarily invertible, there exists a diagram
 \begin{eqnarray}
\begin{aligned}
\xymatrix{
& & K \ar[d]|-{ h} \ar[lld]_-{v} \ar[rrd]^-{u} \ar @{} [ld] | {\stackrel{\lambda} \Leftarrow} 
\ar @{} [rd] | {\stackrel{\rho\cong} \Leftarrow} & &
\\
B   && P \ar[ll]^-{c} \ar[rr]_-{d} & & A \ , }
\end{aligned}
\end{eqnarray}
(with $\lambda$ invertible if and only if $\psi$ is) which pastes onto
\eqref{bipb} to yield $\psi$. This defines on objects an inverse equivalence
of the functor from the category of such $(\lambda,h,\rho)$ to the category of
diagrams \eqref{laxsquare} obtained by pasting onto \eqref{bipb}.    
\end{itemize}
\end{proposition}
\begin{proof}
These are essentially standard facts about groupoid fibrations, especially (i) and (ii). For (iii) we use the groupoid fibration property of $p$ to lift the 2-cell $\psi : nu \Ra pv$ to a 2-cell $\chi : w \Ra v$ with an invertible 2-cell $\nu : nu\Ra pw$ such that 
$\psi = (p\chi)\nu$. Now use the bipullback property of \eqref{bipb} to factor the square $\nu : nu\Ra pw$ as a span morphism $(\sigma, h, \rho) : (w,K,u)\to (c,P,d)$ pasted onto \eqref{bipb}. Then $\lambda$ is the composite of $\sigma$ and $\chi$.    
\end{proof}

The next result is related to Proposition~5 of \cite{8}. 

\begin{proposition}\label{thetahat}
In the bipullback square \eqref{bipb} in the bicategory $\CA$, 
if $p$ is a groupoid fibration and $n$ has a right adjoint $n\dashv u$ 
then $c$ has a right adjoint $c\dashv v$ 
such that the mate $\hat{\theta} : dv \Ra up$ of $\theta : nd\Ra pc$ is invertible.
\end{proposition}
\begin{proof}
Let $\varepsilon : nu \Ra 1_C$ be the counit of $n\dashv u$.
By the groupoid fibration property of $p$, there exists $\chi : w \Ra 1_B$ and an
invertible $\tau : nup\Ra pw$ such that $(p\chi) \tau = \varepsilon p$.
By the bipullback property of \eqref{bipb}, there exists a span morphism  
 \begin{eqnarray*}
\begin{aligned}
\xymatrix{
& & B \ar[d]|-{v} \ar[lld]_-{up} \ar[rrd]^-{ w} \ar @{} [ld] | {\stackrel{\lambda\cong} \Leftarrow} 
\ar @{} [rd] | {\stackrel{\rho\cong} \Leftarrow} & &
\\
A   && P \ar[ll]^-{d} \ar[rr]_-{c} & & B  }
\end{aligned}
\end{eqnarray*}
whose pasting onto $\theta$ is $\tau$.
Then $c\dashv v$ with counit $cv \xRa{\rho^{-1}}w \xRa{\chi}1_B$ and we see that $\hat{\theta} = \rho$ is invertible. 
\end{proof}
 
 \begin{proposition}\label{pbstobipbs}
Suppose $\CC$ is a category with pullbacks. Then the pseudofunctor $(-)_*: \CC\to \mathrm{Spn}\CC$ takes pullbacks to bipullbacks.
\end{proposition}
\begin{proof} Let the span $(p,P,q): A\to B$ be the pullback of the cospan $(f,C,g)$ in $\CC$. Consider a square
 \begin{equation*}
 \begin{aligned}
\xymatrix{
X \ar[d]_{(r,T,s)}^(0.5){\phantom{aaaa}}="1" \ar[rr]^{(u,S,v)}  && A \ar[d]^{f_*}_(0.5){\phantom{aaaa}}="2" \ar@{<=}"1";"2"^-{\psi}_-{\cong}
\\
B \ar[rr]_-{g_*} && C 
}
 \end{aligned}
\end{equation*}  
in $\mathrm{Spn}\CC$. The isomorphism $\psi$ amounts to an isomorphism
$h : (u,S,fv) \to (r,T,gs)$ of spans. In particular, $fv=gsh$, so, by the pullback property, there exists a unique $t : S\to P$ such that $pt=v$ and $qt=sh$.
Then we have a morphism of spans
 \begin{eqnarray*}
\begin{aligned}
\xymatrix{
& & X \ar[d]|-{(u,S,t)} \ar[lld]_-{(r,T,s)} \ar[rrd]^-{ (u,S,v)} \ar @{} [ld] | {\stackrel{\lambda\cong} \Leftarrow} 
\ar @{} [rd] | {\stackrel{\rho\cong} \Leftarrow} & &
\\
B   && P \ar[ll]^-{q_*} \ar[rr]_-{p_*} & & A  }
\end{aligned}
\end{eqnarray*}
in which $\lambda$ is $h: (u,S,qt)\to (r,T,s)$ and $\rho$ is an identity.
\end{proof}

 \section{Lifters}
  
Let $\CM$ be a bicategory. 
 
 We use the notation
\begin{equation}\label{rif}
 \begin{aligned}
   \xymatrix{ & & Y \ar[dd]^{n} \\
     K  \ar[rru]^{\mathrm{rif}(n,u)} \ar[rrd]_{u} & & \dtwocell[0.4]{ll}{\varpi} \\
     & & Z}
 \end{aligned}
\end{equation}
to depict a right lifting $\mathrm{rif}(n,u)$ (see \cite{12}) of $u$ through $n$.
The defining property is that pasting a 2-cell $v\Lra \mathrm{rif}(n,u)$ onto the
triangle to give a 2-cell $nv\Lra u$ defines a bijection. 

A morphism $n :Y\to Z$ is called a {\em right lifter} when $\mathrm{rif}(n,u)$ exists for
all $u : K\to Z$.

\begin{example}
 Left adjoint morphisms in any $\CM$ are right lifters (since the lifting is obtained by composing with the right adjoint).
\end{example}
\begin{example}\label{lifterscompose}
Composites of right lifters are right lifters.  
\end{example}
\begin{example}
 Suppose $\CM = \mathrm{Spn}\CC$ with $\CC$ a finitely complete category. 
 If $f : A\to B$ is powerful (in the sense of \cite{104}, elsewhere called exponentiable, and meaning that the
 functor $\CC/B\to \CC/A$, which pulls back along $f$, has a right adjoint $\Pi_f$) in $\CC$ then $f^* : B\to A$ is a right lifter. The formula is $\mathrm{rif}(f^*,(v,T,q)) = (w,U,r)$ where 
 $$(U\xra{(w,r)}K\times B) = \Pi_{1_K\times f}(T\xra{(v,q)}K\times A) \ .$$ 
\end{example}
\begin{example}\label{m1powerful}
Suppose $m=(m_1,E,m_2)$ is a morphism in $\CM = \mathrm{Spn}\CC$ with $\CC$ a finitely complete category. Then $m$ is a right lifter if and only if $m_1$ is powerful.
The previous Examples imply ``if''. Conversely, we can apply the Dubuc Adjoint Triangle
Theorem (see Lemma 2.1 of \cite{104} for example) to see that $\CM(K,{m_1}^*)$ has a right adjoint
for all $K$ because $\CM(K,m)\cong \CM(K,{m_2}_*)\CM(K,{m_1}^*)$ and the unit of ${m_2}_*\dashv {m_2}^*$ is an equalizer. Taking $K$ to be the terminal object, we conclude that 
$m_1$ is powerful.     
\end{example}

\begin{example}\label{forall}
Let $\CE$ be a regular category and let $\mathrm{Rel}\CE$ be the
locally ordered bicategory of relations in $\CE$ as characterized in \cite{26}. 
The objects are the same as for $\CE$ and the morphisms 
$(r_1,R,r_2) : X\to Y$ are jointly monomorphic spans $X\xla{r_1}R\xra{r_2}Y$
in $\CE$.
Let $\mathrm{Sub}X = \mathrm{Rel}\CE(1,X)$ be the partially ordered set
of subobjects of $X\in \CE$. 
If $f : Y\to X$ is a morphism of $\CE$ then pulling back subobjects of $X$ along $f$
defines an order-preserving function $f^{-1} : \mathrm{Sub}X\to \mathrm{Sub}Y$ whose
right adjoint, if it exists, is denoted by $\forall_f : \mathrm{Sub}Y\to \mathrm{Sub}X$.
A similar analysis as in the span case yields that $(r_1,R,r_2) : X\to Y$ is
a right lifter in $\mathrm{Rel}\CE$ if and only if $\forall_{r_1}$ exists.   
\end{example}

\begin{proposition}\label{exactbipb}
Suppose \eqref{bipb} is a bipullback in $\CM$ with $p$ a groupoid fibration.
If $n$ is a lifter then so is $c$ and, for all morphisms $b : K\to B$, the canonical 2-cell
$$d\circ \mathrm{rif}(c,b)\Lra \mathrm{rif}(n,p\circ b)$$
is invertible.
\end{proposition}
\begin{proof}
For all $K\in \CM$, we have a bipullback square
 \begin{equation*}
 \begin{aligned}
\xymatrix{
\CM(K,P) \ar[d]_{\CM(K,c)}^(0.5){\phantom{aaaaaa}}="1" \ar[rr]^{\CM(K,d)}  && \CM(K,A) \ar[d]^{\CM(K,n)}_(0.5){\phantom{aaaaaa}}="2" \ar@{<=}"1";"2"^-{\cong}
\\
\CM(K,B) \ar[rr]_-{\CM(K,p)} && \CM(K,C) 
}
 \end{aligned}
\end{equation*} 
in $\mathrm{Cat}$ with $\CM(K,p)$ a groupoid fibration and $\CM(K,n)\dashv \mathrm{rif}(n,-)$.
By Proposition~\ref{thetahat}, $\CM(K,c)$ has a right adjoint, so that $c$ is a lifter,
and $\CM(K,d)\mathrm{rif}(c,-)\cong \mathrm{rif}(n,-)\CM(K,p)$.
Evaluating this last isomorphism at $b$ we obtain the isomorphism displayed in the present Proposition. 
\end{proof}

\section{Distributivity pullbacks} 
We now recall Definitions 2.2.1 and 2.2.2 from Weber \cite{Weber2015} of {\em pullback} and {\em distributivity pullback} around a composable pair $(f,g)$ of morphisms in a category $\CC$.
\begin{eqnarray}\label{around}
 \begin{aligned}
\xymatrix{
X \ar[r]^-{p} \ar[d]_-{q} & Z \ar[r]^-{g} & A \ar[d]^-{f}\\
Y \ar[rr]_-{r} & & B
}
 \end{aligned}
\end{eqnarray}
A pullback $(p,q,r)$ around $Z\xra{g}A\xra{f}B$ is a commutative diagram \eqref{around} in which the span $(q,X,gp)$ is a pullback of the cospan $(r,B,f)$ in $\CC$. 

A morphism $t : (p',q',r') \to (p,q,r)$ of pullbacks around $(f,g)$ is a morphism
$t : Y'\to Y$ in $\CC$ such that $rt=r'$. For such a morphism, using the pullback properties, it follows that there is a unique morphism $s : X'\to X$ in $\CC$ such that $ps = p'$ and $qs = tq'$.\footnote {Rather than the single $t$, 
Weber's definition takes the pair $(s,t)$ as
the morphism of pullbacks around $(f,g)$ although he has a typographical error in the condition $rt=r'$.}
This gives a category $\mathrm{PB}(f,g)$. It is worth noting, also using the pullback properties, that the commuting square $qs = tq'$ exhibits the span $(s,X',q')$ as a pullback of the cospan $(q,Y,t)$.

The diagram \eqref{around} is called a distributivity pullback around
$(f,g)$ when it is a terminal object of the category $\mathrm{PB}(f,g)$.   
 
\begin{eqnarray}\label{distributivitypb}
 \begin{aligned}
\xymatrix{
Y \ar[d]_{p_*q^*}^(0.5){\phantom{aaa}}="1" \ar[rr]^{r_*}  && B \ar[d]^{f^*}_(0.5){\phantom{aaa}}="2" \ar@{<=}"1";"2"^-{ }_-{\cong}
\\
Z \ar[rr]_-{g_*} && A 
}
 \end{aligned}
\end{eqnarray}
\begin{proposition}\label{distpbsgivebipbs}
Let $Z\xra{g}A\xra{f}B$ be a composable pair of morphisms in a category $\CC$
with pullbacks.
The diagram \eqref{around} is a pullback around $(f,g)$ in the category $\CC$ if and only if there is a square of the form \eqref{distributivitypb} in the bicategory $\mathrm{Spn}\CC$.
The diagram \eqref{around} is a distributivity pullback around $(f,g)$ in $\CC$ if and only the diagram \eqref{distributivitypb} is a bipullback in $\mathrm{Spn}\CC$.
\end{proposition}
\begin{proof} Passage around the top and right sides of \eqref{distributivitypb} produces the pullback span of the cospan $Y\xra{r}B\xla{f}A$.
Passage around the left and bottom sides produces the left and top sides of
\eqref{around}. That these passages be isomorphic says \eqref{around} is a pullback. 

Suppose \eqref{around} is a distributivity pullback. We will show that
\eqref{distributivitypb} is a bipullback. 
Take a square of the form
 \begin{eqnarray}\label{gensquare}
 \begin{aligned}
\xymatrix{
K \ar[d]_{u_*v^*}^(0.5){\phantom{aaa}}="1" \ar[rr]^{s_*t^*}  && B \ar[d]^{f^*}_(0.5){\phantom{aaa}}="2" \ar@{<=}"1";"2"^-{ }_-{\cong}
\\
Z \ar[rr]_-{g_*} && A 
}
 \end{aligned}
\end{eqnarray}
in $\mathrm{Spn}\CC$. This square amounts to a diagram
 \begin{eqnarray*}
 \begin{aligned}
\xymatrix{
& S\ar[d]^{a} \ar[r]^{u} \ar[ld]_{v}& Z \ar[r]^{g} & A \ar[d]^{f} \\
K & T \ar[l]^{t}\ar[rr]_{s} & & B
}
 \end{aligned}
\end{eqnarray*}
in $\CC$ in which the right-hand region is a pullback around $(f,g)$.
By the distributivity property, there exists a unique pair $(h,k)$ such that
the diagram
  \begin{eqnarray*}
 \begin{aligned}
\xymatrix{
& S\ar[d]^{h} \ar[r]^{a} \ar[ld]_{u}& T \ar[rd]^{s} \ar[d]_{k} &  \\
Z & X \ar[l]^{p}\ar[r]_{q} & Y \ar[r]_{r} & B
}
 \end{aligned}
\end{eqnarray*}
commutes; moreover, the square is a pullback.
Thus we have a span morphism
 \begin{eqnarray*}
\begin{aligned}
\xymatrix{
& & K \ar[d]|-{(t,T,k)} \ar[lld]_-{u_*p^*} \ar[rrd]^-{s_*t^*} \ar @{} [ld] | {\stackrel{\lambda\cong} \Leftarrow} 
\ar @{} [rd] | {\stackrel{\rho\cong} \Leftarrow} & &
\\
Z   && Y \ar[ll]^-{p_*q^*} \ar[rr]_-{r_*} & & B  }
\end{aligned}
\end{eqnarray*}
which pastes onto \eqref{distributivitypb} to yield \eqref{gensquare};
in fact $\rho$ is an identity.
To prove the bipullback 2-cell property, suppose we have span morphisms
$e : (v',S',u')\to (v,S,u)$ and $j : (t',T',s')\to (t,T,s)$ such that composing the first with $g_*$ is the composite of the second with $f^*$.
Then, in obvious notation, $j : (u',a',s') \to (u,a,s)$ is a morphism in $\mathrm{PB}(f,g)$. By the terminal property of $(p,q,r)$, we have $k'=kj$.
This gives the span morphism $j : (t',T',k')\to (t,T,k)$ which is unique as required.

Conversely, suppose \eqref{distributivitypb} is a bipullback. 
We must see that $(p,q,r)$ is terminal in $\mathrm{PB}(f,g)$.
Take another object $(p',q',r')$ of $\mathrm{PB}(f,g)$.
We have the square
\begin{eqnarray*}
 \begin{aligned}
\xymatrix{
Y' \ar[d]_{p'_*q'^*}^(0.5){\phantom{aaa}}="1" \ar[rr]^{r'_*}  && B \ar[d]^{f^*}_(0.5){\phantom{aaa}}="2" \ar@{<=}"1";"2"^-{ }_-{\cong}
\\
Z \ar[rr]_-{g_*} && A 
}
 \end{aligned}
\end{eqnarray*}
which allows us to use the bipullback property to obtain a span morphism
$$(p'_*q'^*,Y',r'_*) \to (p_*q^*,Y,r_*)$$ in $\mathrm{Spn}\CC$ which is compatible with the squares.
Since the span $Y'\to Y$ in this morphism composes with $r_*$ to be isomorphic
to $r'_*$, we see that it has the form $k_*$ for some $k : Y'\to Y$ in $\CC$.
Thus we have our unique $k : (p',q',r')\to (p,q,r)$ in $\mathrm{PB}(f,g)$.            
 \end{proof}

 \section{Polynomials in calibrated bicategories}
 
 Recall from \cite{Ben1967} Section 7 that the {\em Poincar\'e category} $\Pi \CK$
of a bicategory $\CK$ has the same objects as $\CK$, however, the
homset $\Pi \CK(H,K)$ is the set $\pi_0(\CK(H,K))$ of undirected path components of the homcategory
$\CK(H,K)$. Composition is induced by composition of morphisms in $\CK$.
The {\em classifying category} $\mathrm{Cl}\CK$ of $\CK$ is obtained by taking isomorphism classes of morphisms in each category $\CK(H,K)$. 
If $\CK$ is locally groupoidal then $\Pi \CK$ is equivalent to $\mathrm{Cl}\CK$. 

We adapt B\'enabou's notion of ``cat\'egorie calibr\'ee'' \cite{Ben1975} to our present purpose.

\begin{Definition}
A class $\CP$ of morphisms, whose members are called {\em neat} (``propres'' in French), in a bicategory $\CM$ is called a {\em calibration of $\CM$} when it satisfies the following conditions
\begin{itemize}
\item[P0.] all equivalences are neat and, if $p$ is neat and there exists an invertible 2-cell $p\cong q$, then $q$ is neat;
\item[P1.] for all neat $p$, the composite $p\circ q$ is neat if and only if $q$ is neat;
\item[P2.] every neat morphism is a groupoid fibration;
\item[P3.] every cospan of the form $$S\xra{p}Y\xla{n}T \ ,$$ with $n$ a right lifter and $p$ neat, has
a bipullback \eqref{bipb2} in $\CM$ with $\tilde{p}$ neat.
 \begin{equation}\label{bipb2}
 \begin{aligned}
\xymatrix{
P \ar[d]_{\tilde{n}}^(0.5){\phantom{aaaa}}="1" \ar[rr]^{\tilde{p}}  && T \ar[d]^{n}_(0.5){\phantom{aaaa}}="2" \ar@{<=}"1";"2"^-{\theta}_-{\cong}
\\
S \ar[rr]_-{p} && Y 
}
 \end{aligned}
\end{equation}  
\end{itemize} 
A bicategory equipped with a calibration is called {\em calibrated}.

Notice that the class $\mathrm{GF}$ of all groupoid fibrations in any bicategory $\CM$ satisfies all the conditions for a calibration except perhaps the bipullback existence part of P3 (automatically $\tilde{p}$ will be a groupoid fibration by (ii) of Proposition~\ref{bipbproperties}).   

A bicategory $\CM$ is called {\em polynomic} when $\mathrm{GF}$ is a calibration of $\CM$. 

\end{Definition}
 
\begin{Definition} Let $\CM = (\CM,\CP)$ be a calibrated bicategory.
A {\em polynomial} $(m,S,p)$ from $X$ to $Y$ in $\CM$ is a span
$$X\xla{m}S\xra{p}Y$$
in $\CM$ with $m$ a right lifter and $p$ neat.
A {\em polynomial morphism} $(\lambda, h, \rho) : (m,S,p)\to (m',S',p')$ is a diagram
  \begin{eqnarray}\label{polymorph}
\begin{aligned}
\xymatrix{
& & S \ar[d]|-{ h} \ar[lld]_-{m} \ar[rrd]^-{ p} \ar @{} [ld] | {\stackrel{\lambda} \Leftarrow} 
\ar @{} [rd] | {\stackrel{\rho\cong} \Leftarrow} & &
\\
X   && S' \ar[ll]^-{m'} \ar[rr]_-{p'} & & Y  }
\end{aligned}
\end{eqnarray}
in which $\rho$ (but not necessarily $\lambda$) is invertible.
By part (i) of Proposition~\ref{bipbproperties} we know that $h$ must be a groupoid
fibration. (Indeed, by condition P1, $h$ is neat; this is not really needed and is the only use
made herein of the ``only if'' in P1.)
We call $(\lambda, h, \rho)$ {\em strong} when $\lambda$ is invertible. 
A 2-cell $\sigma : h\Ra k : (m,S,p)\to (m',S',p')$ is a 2-cell $\sigma : h\Ra k : S\to S'$ in $\CM$ compatible with $\lambda$ and $\rho$. By Proposition~\ref{gfibsconserve}, we know that $\sigma$ must be invertible.
Write $\mathrm{Poly}\CM(X,Y)$ for the Poincar\'e category of the bicategory of polynomials from $X$ to $Y$ so obtained.
\end{Definition}

We write $\mathbf{h}=[ \lambda, h, \rho ] : (m,S,p)\to (m',S',p')$ 
for the isomorphism class of the polynomial morphism 
$(\lambda, h, \rho) : (m,S,p)\to (m',S',p')$.
We also write $\lambda_h$ and $\rho_h$ when several morphisms are involved.   

\begin{proposition}\label{Spnpolyn}
If $\CC$ is a finitely complete category then the bicategory $\mathrm{Spn}\CC$ is polynomic.
\end{proposition}
\begin{proof} Take a cospan $S\xra{p}Y\xla{n}T$ in $\mathrm{Spn}\CC$ with
$p$ a groupoid fibration and $n = (n_1,F,n_2)$ a lifter.  
By Proposition~\ref{charleftadj}, we can suppose $p$ is actually $p_*$ for
some $p : S\to Y$ in $\CC$. From Example~\ref{m1powerful}, we know that $n_1$
is powerful. Form the pullback span $S\xla{f}P\xra{g}F$ of the cospan
$S\xra{p}Y\xla{n_2}F$ in $\CC$. 
By Proposition~\ref{pbstobipbs}, we have a bipullback
 \begin{equation*}
 \begin{aligned}
\xymatrix{
P \ar[d]_{f_*}^(0.5){\phantom{aaaa}}="1" \ar[rr]^{g_*}  && F \ar[d]^{(n_2)_*}_(0.5){\phantom{aaaa}}="2" \ar@{<=}"1";"2"^-{\cong}
\\
S \ar[rr]_-{p_*} && Y \ . 
}
 \end{aligned}
\end{equation*} 
Since $n_1$ is powerful, Proposition 2.2.3 of Weber \cite{Weber2015} implies we have a distributivity pullback
\begin{eqnarray*}
 \begin{aligned}
\xymatrix{
V \ar[r]^-{a} \ar[d]_-{q} & P \ar[r]^-{g} & F \ar[d]^-{n_1}\\
W \ar[rr]_-{r} & & T
}
 \end{aligned}
\end{eqnarray*}
around $(n_1,g)$. By Proposition~\ref{distpbsgivebipbs}, we have the bipullback
\begin{eqnarray*}
 \begin{aligned}
\xymatrix{
W \ar[d]_{a_*q^*}^(0.5){\phantom{aaa}}="1" \ar[rr]^{r_*}  && T \ar[d]^{(n_1)^*}_(0.5){\phantom{aaa}}="2" \ar@{<=}"1";"2"^-{ }_-{\cong}
\\
P \ar[rr]_-{g_*} && F 
}
 \end{aligned}
\end{eqnarray*}
in $\mathrm{Spn}\CC$. Paste this second bipullback on top of
the first to obtain a bipullback of the cospan $S\xra{p_*}Y\xla{(n_2)_*(n_1)^*}T$
as required. 
\end{proof}

The class of equivalences in any bicategory is a calibration.
In Section~\ref{calibsSpnMod}, we will provide an example of a calibration strictly between equivalences and $\mathrm{GP}$.

In a calibrated bicategory $\CM$, polynomials can be composed as in the diagram
\eqref{comppoly}; this is made possible by Example~\ref{lifterscompose}, Proposition~\ref{exactbipb}, condition P3 and the ``if'' part of condition P1.
Identity spans are also identity polynomials.
\begin{eqnarray}\label{comppoly}
 \begin{aligned}
\xymatrix{
& & \ar @{} [dd] | {\stackrel{\cong \theta} \Longleftarrow} P \ar[ld]_-{\tilde{n}} \ar[rd]^-{\tilde{p}} & & \\
& S \ar[ld]_-{m} \ar[rd]_-{p} & & T \ar[ld]^-{n} \ar[rd]^-{q} & \\
X & & Y & & Z}
 \end{aligned}
\end{eqnarray}
Indeed, this composition of polynomials is the effect on objects of functors
\begin{equation}\label{polycomp}
\circ : \mathrm{Poly}\CM(Y,Z)\times \mathrm{Poly}\CM(X,Y) \lra \mathrm{Poly}\CM(X,Z) \ .
\end{equation}
The effect on morphisms is defined using part (iii) of Proposition~\ref{bipbproperties} as follows. Take morphisms
$\mathbf{h} : (m,S,p) \to (m',S',p')$ and $\mathbf{k} : (n,T,q) \to (n',T',q')$.
 We have a square
\begin{equation*}
 \begin{aligned}
\xymatrix{
P \ar[d]_{h\tilde{n}}^(0.5){\phantom{aaaa}}="1" \ar[rr]^{k\tilde{p}}  && T' \ar[d]^{n'}_(0.5){\phantom{aaaa}}="2" \ar@{<=}"1";"2"^-{\psi}
\\
S' \ar[rr]_-{p'} && Y 
}
 \end{aligned}
\end{equation*} 
in which 
\begin{equation*}
\psi = (n'k\tilde{p}\xRa{\lambda_k\tilde{p}}n\tilde{p}\xRa{\theta \cong} p\tilde{n}\xRa{\rho_h\tilde{n}\cong}p'h\tilde{n}) \ .
\end{equation*}
Now we use Proposition~\ref{bipbproperties} to obtain, in obvious primed notation, a diagram
 \begin{eqnarray*}
\begin{aligned}
\xymatrix{
& & P \ar[d]|-{\ell} \ar[lld]_-{h\tilde{n}} \ar[rrd]^-{k\tilde{p}} \ar @{} [ld] | {\stackrel{\sigma} \Leftarrow} 
\ar @{} [rd] | {\stackrel{\tau \cong} \Leftarrow} & &
\\
S'   && P' \ar[ll]^-{\tilde{n}'} \ar[rr]_-{\tilde{p}'} & & T'  }
\end{aligned}
\end{eqnarray*} 
which leads to the polynomial morphism
\begin{equation*}
((\lambda_h\tilde{n})(m'\sigma)), \ell, (q'\tau)(\rho_k\tilde{p}) : (m\tilde{n},P,q\tilde{p})\lra (m'\tilde{n}',P',q'\tilde{p}') 
\end{equation*}
whose isomorphism class is the desired 
$$\mathbf{k}\circ \mathbf{h} : (n,T,q)\circ (m,S,p)\to (n',T',q')\circ (m',S',p') \ .$$ 

\begin{proposition}
There is a bicategory $\mathrm{Poly}\CM$ of polynomials in a calibrated bicategory $\CM$. The objects are those of $\CM$, the homcategories
are the $\mathrm{Poly}\CM(X,Y)$. Composition is given by the functors \eqref{polycomp}.
The vertical and horizontal stacking properties of bipullbacks provide the associativity isomorphisms.
\end{proposition}

We write $\mathrm{Poly}_{\mathrm{s}}\CM$ for the sub-bicategory of 
$\mathrm{Poly}\CM$ obtained by restricting to the strong polynomial morphisms.

\begin{example}
If $\CC$ is a finitely complete category then the bicategory $\mathrm{Poly}\mathrm{Spn}\CC$ is biequivalent to the bicategory denoted by $\mathrm{Poly}_{\CC}$ in Gambino-Kock \cite{GambinoKock} and by $\mathrm{Poly}({\CC})$ in Walker \cite{Walker2018}.
Moreover, $\mathrm{Poly}_{\mathrm{s}}\mathrm{Spn}\CC$
is biequivalent to Walker's bicategory $\mathrm{Poly}_{\mathrm{c}}(\CC)$.
Note that the isomorphism classes $\mathbf{h}$ of polynomial morphisms have canonical representatives of the form $f_*$ (since each span $(u,S,v) : U \to V$ with $u$ invertible is isomorphic to $(1_U, U, v \ u^{-1})$). 
\end{example}

\begin{proposition}\label{bbH}
If the bicategory $\CM$ is calibrated then, for each $K\in \CM$, there is a pseudofunctor
$\mathbb{H}_K : \mathrm{Poly}\CM \lra \mathrm{Cat}$ taking the polynomial $X\xla{m} S\xra{p} Y$
to the composite functor 
$$\CM(K,X)\xra{\mathrm{rif}(m,-)} \CM(K,S) \xra{\CM(K,p)} \CM(K,Y) \ .$$
The 2-cell $\mathbf{h} : (m,S,p)\to (n,T,q)$ in $\mathrm{Poly}\CM$ is taken to the natural transformation obtained by the pasting 
\begin{equation*}
 \begin{aligned}
   \xymatrix{ & &  \CM(K,S) \ar[dd]|-{ \CM(K,h)} \ar[rrd]^{ \CM(K,p)} & &\\
     \CM(K,X)  \ar[rru]^{\mathrm{rif}(m,-)} \ar[rrd]_{\mathrm{rif}(n,-)} & &  \dtwocell[0.4]{ll}{\hat{\lambda}} & & \dtwocell[0.7]{ll}{\CM(K,\rho)}  \CM(K,Y) \\
     & &  \CM(K,T)  \ar[rru]_{ \CM(K,q)} & &}
 \end{aligned}
\end{equation*}
where $\hat{\lambda}$ is the mate under the adjunctions of the natural transformation $\CM(K,\lambda) : \CM(K,n)\CM(K,h)\Ra \CM(K,m)$.  
\end{proposition}
\begin{proof}
We will show that polynomial 2-cells 
$\alpha : (\lambda_h,h,\rho_h)\Ra (\lambda',h',\rho') : (m,S,p)\to (n,T,g)$ are taken to identities.
Since $$\CM(K,\lambda) = \Big(\CM(K,n)\CM(K,h)\xRa{\CM(K,n)\CM(K,\alpha)}\CM(K,n)\CM(K,h')\xRa{\CM(K,\lambda')}\CM(K,m)\Big) \ ,$$
it follows that $$\hat{\lambda} = \Big(\CM(K,h)\mathrm{rif}(m,-)\xRa{\CM(K,\alpha)\mathrm{rif}(m,-)}\CM(K,h')\mathrm{rif}(m,-)\xRa{\hat{\lambda}'}\mathrm{rif}(n,-)\Big) \ .$$
Using this and that $\rho' = (g\alpha)\rho$, we have the identity $$(g\hat{\lambda})(\rho \ \mathrm{rif}(m,u)= (g\hat{\lambda}')(\rho' \ \mathrm{rif}(m,u) : f \ \mathrm{rif}(m,u)\Lra g \ \mathrm{rif}(m,u)$$ induced by $\alpha$ as claimed.  

That $\mathbb{H}_K$ is a pseudofunctor follows from Proposition~\ref{exactbipb}.
\end{proof}

We can put somewhat more structure on the image of the pseudofunctor $\mathbb{H}_K$.
Recall the definition (for example, in \cite{McCrud2000} Section 3) of the 2-category 
$\CV\text{-}\mathrm{Act}$ of $\CV$-actegories for a monoidal category $\CV$.
 
Composition in $\CM$ yields a monoidal structure on the category $\CV_K = \CM(K,K)$
and a right $\CV_K$-actegory structure on each category $\mathbb{H}_KX=\CM(K,X)$:
\begin{eqnarray*}
-\circ- : \CM(K,X)\times \CM(K,K) \lra \CK(K,X) \ .
\end{eqnarray*}
We can replace the codomain $\mathrm{Cat}$ of $\mathbb{H}_K$ in Proposition~\ref{bbH} by $\CV_K\text{-}\mathrm{Act}$.   

To see this, we need a $\CV_K$-actegory morphism structure on each functor 
$$\mathbb{H}_K(m,S,p) = p \ \mathrm{rif}(m,-) : \CM(K,X)\to \CM(K,Y) \ .$$ 
However, for each $a\in \CV_K$ and $u\in \CM(K,X)$,
 we have the canonical $m\mathrm{rif}(m,u)a\xRightarrow{\varpi a} ua$ which induces a 2-cell
  $\mathrm{rif}(m,u)a\Ra \mathrm{rif}(m,ua)$. 
  Whiskering this with $p : S \to Y$, we obtain the component at $(u,a)$ of a natural transformation:
  \begin{equation*}
 \begin{aligned}
\xymatrix{
\CM(K,X)\times \CV_K \ar[d]_{p \ \mathrm{rif}(m,-)\times 1_{\CV_K}}^(0.5){\phantom{aaaaaaaa}}="1" \ar[rr]^{-\circ -}  && \CM(K,X) \ar[d]^{p \ \mathrm{rif}(m,-)}_(0.5){\phantom{aaaaaaaa}}="2" \ar@{=>}"1";"2"^-{}
\\
\CM(K,Y)\times \CV_K \ar[rr]_-{-\circ -} && \CM(K,Y) \ . 
}
 \end{aligned}
\end{equation*} 
The axioms for an actegory morphism are satisfied and each $\mathbb{H}_K(m,S,p)$ is a 2-cell
in $\CV_K\text{-}\mathrm{Act}$.  

In fact, we have a pseudofunctor
$$\mathbb{H} : \mathrm{Poly}\CM \lra \mathrm{Hom}(\CM^{\mathrm{op}},\mathrm{Act})$$
where $\mathrm{Act}$ is the 2-category of pairs $(\CV,\CC)$ consisting of a monoidal
category $\CV$ and a category $\CC$ on which it acts.

\section{Bipullbacks from tabulations}

Tabulations in a bicategory, in the sense intended here, appeared in \cite{26} to characterize bicategories of spans.

For any bicategory $\CM$, we write $\CM_*$ for the sub-bicategory obtained by restricting to left adjoint morphisms.
For each left adjoint morphism $f:X\to Y$ in $\CM$, we write $f^*: Y\to X$ for a right adjoint.

\begin{Definition} The bicategory $\CM$ is said to {\em have tabulations from the terminal} when the following conditions hold:
\begin{itemize} 
\item[(i)] the bicategory $\CM_*$ has a terminal object $1$ with the property
that, for all objects $U$, the unique-up-to-isomorphism left-adjoint morphism 
$!_U : U\to 1$ is terminal in the category $\CM(U,1)$;
\item[(ii)] for each morphism $u : 1\to X$ in $\CM$, there is a diagram \eqref{tabu}, 
called a {\em tabulation} of $u$, in which $p : U\to X$ is a left adjoint morphism and such that the diagram
\begin{eqnarray}\label{tabunivprop}
 \begin{aligned}
\xymatrix{
\CM(K,U) \ar[d]_{\CM(K,p)}^(0.5){\phantom{aaaaaaaa}}="1" \ar[rr]^{ }  && \mathbf{1} \ar[d]^{\lceil u!_K \rceil}_(0.5){\phantom{aaaaaaaa}}="2" \ar@{=>}"1";"2"^-{\lambda}
\\
\CM(K,X) \ar[rr]_-{1_{\CM(K,X)}} && \CM(K,X) \ , 
}
 \end{aligned}
 \end{eqnarray}
where the natural transformation $\lambda$ has component 
$pw\xra{\rho_u w}u!_Uw\xra{u!}u!_K$ at $w\in \CM(K,U)$, exhibits $\CM(K,U)$
as a bicategorical comma object in $\mathrm{Cat}$.
\end{itemize}
\end{Definition}  
\begin{equation}\label{tabu}
\begin{aligned}
\xymatrix{
  & U\ar[ld]_{!_U}^(0.5){\phantom{a}}="1" \ar[rd]^{p}_(0.5){\phantom{a}}="2" \ar@{<=}"1";"2"^-{\rho_u} & 
\\
1 \ar[rr]_-{u} && X 
}
\end{aligned}
\end{equation}
\begin{remark}\label{tabrmk}
\begin{itemize}
\item[(a)] The bicategorical comma property of the diagram \eqref{tabunivprop} implies $p$ is an er-fibration in $\CM$. 
\item[(b)] Notice that condition (ii) in this Definition does agree with combined conditions T1 and 
T2 in the definition of tabulation in \cite{26} for morphisms with domain $1$. This is because all left adjoints $K\to 1$ are isomorphic to $!_K$ using condition (i) of our Definition.
\item[(c)] Using (b) and Proposition~1(d) of \cite{26}, we see that the mate $p!_U^* \Ra u$ of $\rho_u : p \Ra u !_U$ is invertible.  
Let us denote the unit of the adjunction $\ !_U \dashv \ !_U^*$ by $\eta_U : 1_U \Lra \ !_U^* \ !_U$.
So we can replace $u$ up to isomorphism by $p!_U^*$ and $\rho_u$ by $p\eta_U$.
\item[(d)] If $\CM$ has tabulations from the terminal and we have a morphism $p : U \to X$
such that \eqref{tabunivprop} has the bicategorical comma property with $u=p!_U^* : 1\to X$
and $\rho_u=p\eta_U$ then $p$ is a left adjoint.
This is because a tabulation of $u : 1\to X$ does exist in which the right leg is a left adjoint
and the comma property implies the right leg is isomorphic to $pe$ for some equivalence $e$. 
\item[(e)] Another way to express the comma object condition \eqref{tabunivprop} is to say \eqref{tabunivpropbipb} is a bipullback for $\hat{\lambda}w = \lambda_w = (pw\xra{\rho_u w}u!_Uw\xra{u!}u!_K)$.
\begin{equation}\label{tabunivpropbipb}
\begin{aligned}
\xymatrix{
\CM(K,U) \ar[d]_{\hat{\lambda}}^(0.5){\phantom{aaaaaaaa}}="1" \ar[rr]^{!}  && \mathbf{1} \ar[d]^{\lceil u!_K \rceil}_(0.5){\phantom{aaaaaaaa}}="2" \ar@{=}"1";"2"^-{}
\\
\CM(K,X)^{\mathbf{2}} \ar[rr]_-{\mathrm{cod}} && \CM(K,X)  
}
\end{aligned}
\end{equation}
\end{itemize}
\end{remark}  
Let $\mathrm{Tab}$ be the class of morphisms 
$p : U\to X$ in $\CM$ which occur in a tabulation $\eqref{tabu}$.

\begin{theorem}\label{tabuthm}
The class $\mathrm{Tab}$ is a calibration for any bicategory $\CM$ which has tabulations from the terminal. 
\end{theorem}
\begin{proof} We must prove properties P0--P3 for a calibration.
Property P0 is obvious. For P1, take $V\xra{q}U\xra{p}X$ with $p\in\mathrm{Tab}$.
By (c) and (d) of Remark~\ref{tabrmk}, $p$ can be assumed to come from the tabulation of $u = p!^*_U$.
If $q$ is to come from a tabulation it must be of $v= q!^*_V$.
If $pq$ is to come from a tabulation it must be of $w= pq!^*_V= pv$.
Contemplate the following diagram in which the $\hat{\lambda}$ comes from $v$.
\begin{equation*}
\begin{aligned}
\xymatrix{
\CM(K,V) \ar[d]_{\hat{\lambda}}^(0.5){\phantom{aaaaaaaa}}="1" \ar[rr]^{!}  && \mathbf{1} \ar[d]^{\lceil q!_V^*!_K \rceil}_(0.5){\phantom{aaaaaaaa}}="2" \ar@{=}"1";"2"^-{}
\\
\CM(K,U)^{\mathbf{2}} \ar[rr]_-{\mathrm{cod}} \ar[d]_-{\CM(K,p)^{\mathbf{2}}} && \CM(K,U) \ar[d]^-{\CM(K,p)}
\\
\CM(K,X)^{\mathbf{2}} \ar[rr]_-{\mathrm{cod}} && \CM(K,X)  
}
\end{aligned}
\end{equation*}
By Remark~\ref{tabrmk}(a), $p$ is a groupoid fibration; incidentally, this gives P2. 
So the bottom square is a bipullback (see Proposition~\ref{dfib_cotensor}).
Therefore, the top square is a bipullback if and only if the pasted square is a bipullback.
By Remark~\ref{tabrmk}(d) and (e), this says $q\in \mathrm{Tab}$ if and only if $pq\in \mathrm{Tab}$. This proves P1.

It remains to prove P3. 
We start with a cospan
$Z\xra{p}C\xla{m}B$ with $m$ a right lifter and $p\in \mathrm{Tab}$.
Put $z = p !_Z^* : 1\to C$ and tabulate $y = \mathrm{rif}(m,z) : 1\to B$ 
as $y = r \ !_Y^*$ for $r : Y\to B$ in $\mathrm{Tab}$.
Using the tabulation property of $Z$, we induce $n$ and invertible $\theta$ as in the diagram \eqref{construction} in which the triangle containing $\varpi$ exhibits the
right lifting $\mathrm{rif}(m,z)$.
\begin{equation}\label{construction}
\begin{aligned}
\xymatrix{\ar @{} [rdd] | {\stackrel{\cong \theta} \Longrightarrow}
Y \ar@/^/^-{!_Y} [rrrd] \ar[rd]_{n} \ar[dd]_-{r}   \ar @{} [rrd] | {\stackrel{ } \Longrightarrow}  &  & \\
& Z \ar[rr]^{!_Z} \ar[rd]_{p} & & 1 \ar[ld]^{z} \\
B \ar[rr]_{m} & & C  &  \ar @{} [llu] | {\stackrel{g\eta_Z} \Longrightarrow} 
}
\qquad
\xymatrix{
\\
{=}}
\qquad
\xymatrix{\ar @{} [rd] | {\stackrel{r\eta_Y}\Lra}
Y \ar[rr]^-{!_Y}  \ar[dd]_-{r} &  & 1 \ar[dd]^-{z} \ar[lldd]^-{y}   \\
& &
\\
B  \ar[rr]_{m} & & C  \ar @{} [lu] | {\stackrel{\varpi} \Longrightarrow}    
}
\end{aligned}
\end{equation}
It is the region containing $\theta$ in \eqref{construction} that we will show is a bipullback.
For all $K\in \CM$, we must show that the left-hand square in the diagram
\begin{eqnarray*}
 \begin{aligned}
\xymatrix{
\CM(K,Y) \ar[d]_{\CM(K,r)}^(0.5){\phantom{aaaaaaaa}}="1" 
\ar[rr]^{\CM(K,n)}  
&& \CM(K,Z) \ar[rr]^{} 
     \ar[d]_{\CM(K,p)}_(0.5){\phantom{aaaaaaaa}}="2" \ar@{=>}"1";"2"^-{\cong}_-{} 
     \ar@{}[d]^{\phantom{aaaaaaaa}}="3"
&& \mathbf{1} \ar[d]^{\lceil z!_K \rceil}_(0.5){\phantom{aaaaaaaa}}="4" \ar@{=>}"3";"4"^-{\lambda}_-{}
\\
\CM(K,B) \ar[rr]_-{\CM(K,m)} && \CM(K,C) \ar[rr]_-{1_{\CM(K,C)}} && \CM(K,C) 
}
 \end{aligned}
 \end{eqnarray*}
is a bipullback.
However, the right-hand square has the comma property. 
So the bipullback property of the left-hand square is equivalent to the comma property of the pasted diagram.
However, using \eqref{construction}, we see that the pasted composite is equal to
the pasted composite
\begin{eqnarray*}
 \begin{aligned}
\xymatrix{
\CM(K,Y) \ar[d]_{\CM(K,r)}^(0.5){\phantom{aaaaaaaa}}="1" 
\ar[rr]^{}  
&& \mathbf{1} \ar[rr]^{} 
     \ar[d]_{\lceil y!_K \rceil}_(0.5){\phantom{aaaaaaaa}}="2" \ar@{=>}"1";"2"^-{\lambda}_-{} 
     \ar@{}[d]^{\phantom{aaaaaaaa}}="3"
&& \mathbf{1} \ar[d]^{\lceil z!_K \rceil}_(0.5){\phantom{aaaaaaaa}}="4" \ar@{=>}"3";"4"^-{\lceil \varpi !_K \rceil}_-{}
\\
\CM(K,B) \ar[rr]_-{1_{\CM(K,B)}} && \CM(K,B) \ar[rr]_-{\CM(K,m)} && \CM(K,C) 
}
 \end{aligned}
 \end{eqnarray*}
Here, the left-hand square has the comma property and $y!_K$ is the value of the right
adjoint $\mathrm{rif}(m,-)$ to $\CM(K,m)$ at $z!_K$. So the pasted composite does
have the comma property, as required.
\end{proof}

\section{Calibrations of $\mathrm{Spn}\CC$, of $\mathrm{Rel}\CE$ and of $\mathrm{Mod}$}\label{calibsSpnMod}

If $\CC$ is a category with finite limits, its terminal object $\mathbf{1}$ clearly
has the property (i) in the definition of tabulations from the terminal.
Then, from Remark~\ref{tabrmk} and \cite{26}, we know that $\mathrm{Spn}\CC$ 
has tabulations from the terminal. 

A span from $\mathbf{1}$ to $X$ has the form $(!,U,p) : \mathbf{1}\to X$
for some $p : U\to X$ in $\CC$. A tabulation of the span 
is provided by the diagram
\begin{equation*}
\begin{aligned}
\xymatrix{
  & U\ar[ld]_{!_{U *}}^(0.5){\phantom{a}}="1" \ar[rd]^{p_*}_(0.5){\phantom{a}}="2" \ar@{<=}"1";"2"^-{\rho_u} & 
\\
1 \ar[rr]_-{(!,U,p)} && X  \ .
}
\end{aligned}
\end{equation*}
It follows that $\mathrm{Tab}$ consists of spans of the form $p_*$ for some
morphism $p$ in $\CC$. Using Proposition~\ref{charleftadj}, we deduce:
\begin{proposition}
For the bicategory $\mathrm{Spn}\CC$ of spans in a finitely complete category 
$\CC$, $\mathrm{Tab} = \mathrm{GF}$.  
\end{proposition}
With this, Theorem~\ref{tabuthm} provides another proof of Proposition~\ref{Spnpolyn}.
\bigskip

Here is the result in the case of bicategories of relations.

\begin{proposition}
The bicategory $\mathrm{Rel}\CE$ of relations in a regular category 
$\CE$ has tabulations from the terminal. Moreover the calibration $\mathrm{Tab}$
of $\mathrm{Rel}\CE$ consists of those relations isomorphic to $p_*$
for some monomorphism $p$ in $\CE$.  
\end{proposition}
\begin{proof}
The existence of tabulations was shown in \cite{26}.
The right leg of a tabulation of $(!,R,p) : 1\to X$ is of course $p: R\to X$
which must be a monomorphism for the span $(!,R,p)$ to be a relation.  
\end{proof}
With this and Example~\ref{forall}, we obtain a different notion of ``polynomials'' 
in a regular category; again they are the morphisms of a bicategory $\mathrm{PolyRel}\CE$.

\begin{example}\label{toposex} 
An elementary topos $\CE$ admits two basic constructions, the power object
(or relation classifier) $\mathcal{P}X$ and the partial map classifier $\widetilde{X}$; see \cite{ii, JohnstoneTT}.
Both define object assignments for monads on $\CE$. There is a distributive law 
$d_X : \mathcal{P}\widetilde{X} \to \widetilde{\mathcal{P}X}$ between the two monads.   
We claim that, for a topos $\CE$, the classifying category of $\mathrm{PolyRel}\CE$ is equivalent 
to the opposite of the Kleisli category $\CE_{\widetilde{\mathcal{P}(-)}}$ for the composite monad 
$X\mapsto \widetilde{\mathcal{P}X}$. To see this, we need some detail on the monads involved.

The (covariant) power endofunctor $\mathcal{P}$ on $\CE$ is defined on morphisms $u : X\to Y$ by direct
image $\exists_u : \mathcal{P}X\to \mathcal{P}Y$. The partial map classifier takes $u$
to $\widetilde{u} :\widetilde{X}\to \widetilde{Y}$ corresponding to the partial map 
$u : \widetilde{X}\to Y$ which is $u$ with $X$ as domain of definition. 
The unit $\sigma : 1_{\CE}\Lra \mathcal{P}$ for the monad $\mathcal{P}$ has
components $\sigma_X : X\Lra \mathcal{P}X$ corresponding to the identity relation on $X$.
Similarly, the unit $\eta : 1_{\CE}\Lra \widetilde{(-)}$ for the monad $\widetilde{(-)}$ has
components $\eta_X : X\Lra \widetilde{X}$ corresponding to the identity partial map on $X$.     

Rather than examine the multiplications for these monads, we take the 
``no iteration'' or ``mw-'' point of view (see \cite{ManesAT, ACU, MWn-it, 121}) from which
the Kleisli bicategory is easily obtained.
For $\mathcal{P}$, the extra data needed are functions
$$\CE(X,\mathcal{P}Y)\lra \CE(\mathcal{P}X,\mathcal{P}Y) \ ;$$  
they take $X\xra{f}\mathcal{P}Y$ to the supremum-preserving extension
$\mathcal{P}X\xra{f'}\mathcal{P}Y$ of $f$ along $\sigma_X$.
The Kleisli category for $\mathcal{P}$ is the classifying category $\mathrm{ClRel}\CE$ of the bicategory of relations in $\CE$.  
For $\widetilde{(-)}$, the extra data needed are functions
$$\CE(X,\widetilde{Y})\lra \CE(\widetilde{X},\widetilde{Y}) \ ;$$  
they take $X\xra{f}\widetilde{Y}$ to the bottom-preserving extension
$\widetilde{X}\xra{f_1}\widetilde{Y}$ of $f$ along $\eta_X$.
The Kleisli category for $\widetilde{(-)}$ is the classifying category $\mathrm{ClPar}\CE$ 
of the bicategory $\mathrm{Par}\CE$ of partial maps in $\CE$: 
it is the subbicategory of $\mathrm{Spn}\CE$
whose morphisms are restricted to those spans $X\xla{i}U\xra{f}Y$ for which the left leg $i$
is a monomorphism.  

To give a distributive law $d_X : \mathcal{P}\widetilde{X} \to \widetilde{\mathcal{P}X}$
is equally to give a lifting $\widehat{\mathcal{P}}$ of the monad $\mathcal{P}$ to a monad on the
Kleisli category $\mathrm{ClPar}\CE$ of $\widetilde{(-)}$. Indeed, we can lift $\mathcal{P}$
to a pseudomonad $\widehat{\mathcal{P}}$ on $\mathrm{Par}\CE$. We use the facts that
$\mathcal{P}$ preserves pullbacks of monomorphisms along arbitrary morphisms and that the square
\begin{eqnarray*}
\xymatrix{
U \ar[rr]^-{i} \ar[d]_-{\sigma_U} && X \ar[d]^-{\sigma_X} \\
\mathcal{P}U \ar[rr]_-{\exists_i} && \mathcal{P}X}
\end{eqnarray*}
is a pullback when $i$ is a monomorphism. These imply that we can define $\widehat{\mathcal{P}}$
on objects to be $\mathcal{P}$ and on partial maps by 
$$\widehat{\mathcal{P}}(X\xla{i}U\xra{f}Y) = (\mathcal{P}X\xla{\exists_i}\mathcal{P}U\xra{\exists_f}\mathcal{P}Y)$$ 
to obtain a pseudofunctor, and that $X\xla{1_X}X\xra{\sigma_X}\mathcal{P}X$ provides a pseudonatural unit.    
Again, rather than a multiplication for $\widehat{\mathcal{P}}$, we supply the functor
\begin{eqnarray*}
\mathrm{Par}\CE(X,\mathcal{P}Y)\lra \mathrm{Par}\CE(\mathcal{P}X,\mathcal{P}Y) \ , \ (X\xla{i}U\xra{f}\mathcal{P}Y) \mapsto (\mathcal{P}X\xla{\exists_i}\mathcal{P}U\xra{f'}\mathcal{P}Y) \ .
\end{eqnarray*}
The Kleisli category $\CE_{\widetilde{\mathcal{P}(-)}}$ of the composite monad $\widetilde{\mathcal{P}(-)}$ on $\CE$ is the classifying category
for the Kleisli bicategory $(\mathrm{Par}\CE)_{\widehat{\mathcal{P}}}$ of the pseudomonad 
$\widehat{\mathcal{P}}$ on $\mathrm{Par}\CE$. 

The claim at the beginning of this example will follow after we see that $\mathrm{PolyRel}\CE$ is biequivalent
to the opposite of $(\mathrm{Par}\CE)_{\widehat{\mathcal{P}}}$.  
To see this, notice that the objects of the two bicategories are the same: they are the objects of $\CE$.
Also, we have the pseudonatural equivalence
\begin{eqnarray*}
\mathrm{PolyRel}\CE(X,C) \simeq \mathrm{Par}\CE(C,\mathcal{P}X)
\end{eqnarray*}
of hom categories under which the polynomial $X\xla{(a_1,A,a_2)} Z\xra{(1_Z,Z,p)} C$
corresponds to the partial map $C\xla{p}Z\xra{a} \mathcal{P}X$ where $a$ classifies the
relation $(a_1,A,a_2)$. 

What remains is to see that span composition of polynomials transports to Kleisli composition. 
We shall write for the case $\CE = \mathrm{Set}$ and appeal to topos internal logic
to justify the argument in general.
First we look at composition in $\mathrm{PolyRel}\CE$. So that we can make use of the
notation in the construction of pseudopullback in \eqref{construction}, 
we look at the following span composite. 
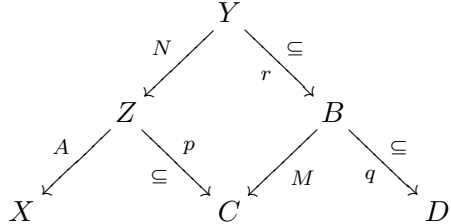
\begin{eqnarray*}
\xymatrix{
& & Y\ar[rd]^-{\subseteq}_-{r} \ar[ld]_-{N} & & \\
& Z \ar[rd]_-{\subseteq}^-{p} \ar[ld]_-{A} &  & B\ar[rd]^-{\subseteq}_-{q} \ar[ld]^-{M}& \\
X & & C & & D }
\end{eqnarray*}
The object $Y$ and relation $N$ are obtained from the subobjects $Z\subseteq C$ and $B\subseteq D$, 
and relations $A$ and $M$. Referring to \eqref{construction}, we see that
\begin{eqnarray*}
Y = \{b\in B : bMc  \text{ implies }  c\in Z \}
\end{eqnarray*}
and $N$ is the restriction of the relation $M$.
Now we look at composition of the corresponding morphisms in the Kleisli bicategory;
this is given by the diagram
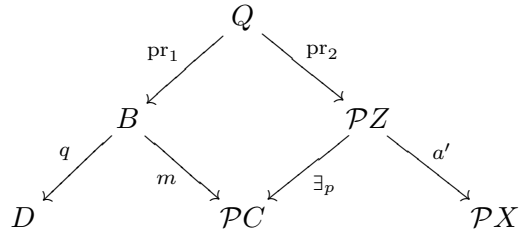
\begin{eqnarray*}
\xymatrix{
& & Q\ar[rd]^-{\mathrm{pr}_2} \ar[ld]_-{\mathrm{pr}_1} & & \\
& B \ar[rd]_-{m} \ar[ld]_-{q} &  & \mathcal{P}Z\ar[rd]^-{a'} \ar[ld]^-{\exists_p}& \\
D & & \mathcal{P}C & & \mathcal{P}X }
\end{eqnarray*}
in which the diamond is a pullback while $m$ and $a$ classify the relations $M$ and $A$.
We therefore have an isomorphism
\begin{eqnarray*}
Q = \{b\in B : m(b) \subseteq Z \} \cong Y
\end{eqnarray*}    
under which $q\circ \mathrm{pr}_1$ and $a'\circ \mathrm{pr}_2$ transport to $q\circ r$ and 
the classifier of $A\circ N$.

Incidentally, using this biequivalence, we can view the pseudofunctor $\mathbb{H}_K$
of Proposition~\ref{bbH} as a pseudofunctor
$$(\mathrm{Par}\CE)_{\widehat{\mathcal{P}}}^{\mathrm{op}}\lra \mathrm{Ord}$$
into ordered sets taking $C\xla{p}Z\xra{a} \mathcal{P}X$ to the order-preserving function
\begin{eqnarray*}
\mathrm{Rel}\CE(K,X)\xra{\mathrm{rif}(a,-)}\mathrm{Rel}\CE(K,Z)\xra{p\circ -}\mathrm{Rel}\CE(K,C)     
\end{eqnarray*}
whose value at a relation $(s_1,S, s_2) : K\to X$ is the relation $(c,a/s, p\circ d) : K\to C$ as in the diagram
\begin{eqnarray*}
\xymatrix{
&& a/s \ar[lld]_{p\circ d} \ar[d]_{d}^(0.5){\phantom{aaaaa}}="1" \ar[rr]^{c}  && K \ar[d]^{s}_(0.5){\phantom{aaaaa}}="2" \ar@{=>}"1";"2"^-{\leq}
\\
C && Z \ar[rr]_-{a}  \ar[ll]^-{p} && \mathcal{P}X 
}
\end{eqnarray*}
in which the square has the comma property and $s$ classifies the relation $(s_1,S, s_2)$.     
\end{example}
\bigskip

Next we look at the bicategory $\mathrm{Mod} = \CV\text{-}\mathrm{Mod}$,
where $\CV = \mathrm{Set}$; see Example~\ref{powersinMod}. 

\begin{proposition}
The bicategory $\mathrm{Mod}$ has tabulations from the terminal. 
\end{proposition}
\begin{proof}
We need to check the validity of conditions (i) and (ii) defining the having of tabulations from the terminal.
Since the terminal object $\mathbf{1}$ of $\mathrm{Cat}$ is Cauchy complete (idempotents split) \cite{LawMetric}, every left-adjoint module $K \to \mathbf{1}$ is isomorphic
to $!_{K *} : K \to \mathbf{1}$ where $!_{K} : K \to \mathbf{1}$ is the unique functor. 
The module $!_{K *}$, as a functor $\mathbf{1}^{\mathrm{op}}\times K\to \mathrm{Set}$, is constant at a one-point set. So condition (i) holds. 

For condition (ii), take a module $u : \mathbf{1} \to X$ regarded as a functor
$u : X^{\mathrm{op}}\to \mathrm{Set}$.
Form the comma category $U$ of $u$ as in the square
 \begin{eqnarray}\label{elofu}
 \begin{aligned}
\xymatrix{
U \ar[d]_{p}^(0.5){\phantom{aaaaaa}}="1" \ar[rr]^{ }  && \mathbf{1} \ar[d]^{\lceil u\rceil}_(0.5){\phantom{aaaaaa}}="2" \ar@{=>}"1";"2"^-{\rho_u}
\\
X \ar[rr]_-{\mathrm{yon}_X} && [X^{\mathrm{op}}, \mathrm{Set}] \ . 
}
 \end{aligned}
 \end{eqnarray}
 The natural transformation in the square has components $\rho_{u x} : X(x,p-)\to ux$ which reinterprets as a 2-cell
 \begin{equation*}
\begin{aligned}
\xymatrix{
  & U\ar[ld]_{!_{U *}}^(0.5){\phantom{a}}="1" \ar[rd]^{p_*}_(0.5){\phantom{a}}="2" \ar@{<=}"1";"2"^-{\rho_u} & 
\\
1 \ar[rr]_-{u} && X 
}
\end{aligned}
\end{equation*}
in $\mathrm{Mod}$.
In fact, we see that $U\xra{p} X$ is $\wr u$ in the sense of Proposition~\ref{Eler}.
So $\wr u!_{K *}$ is $ K^{\mathrm{op}} \times U \xra{1_{K^{\mathrm{op}}}\times p} K^{\mathrm{op}}\times X$. Using Proposition~\ref{Eler}, we see that the comma construction $\mathrm{Mod}(K,X)/\lceil u !_{K *}\rceil$ is biequivalent to 
$$\mathrm{ERFib}(K^{\mathrm{op}}\times X)/(K^{\mathrm{op}} \times U \xra{1_{K^{\mathrm{op}}}\times p} K^{\mathrm{op}}\times X)
\sim \mathrm{ERFib}(K^{\mathrm{op}}\times U)$$
and, again by Proposition~\ref{Eler}, this is biequivalent to  $\mathrm{Mod}(K,U)$, as required for the comma property of diagram \eqref{elofu}.    
\end{proof}
\begin{corollary}
The bicategory $\mathrm{Mod}$ is calibrated by $\mathrm{Tab}$.
All morphisms are right lifters and,
up to equivalence, the neat morphisms are those of the form $p_* : E\to B$ 
where $p$ is a discrete fibration. 
\end{corollary}
\begin{example}\label{Modex}
The bicategory $\mathrm{PolyMod}$ is biequivalent 
to the opposite of the Kleisli bicategory for the composite 
$X\mapsto \mathrm{Fam}^\mathrm{op}[X^{\mathrm{op}},\mathrm{Set}]$ of the colimit-completion pseudomonad and the product-completion pseudomonad (modulo obvious size issues).

To see this, note that the coproduct completion $\mathrm{Fam}X$ of a category $X$ can be efficiently described, in the terminology
of Section 4 of \cite{ocad}, as the lax comma object
\begin{equation*}
\xymatrix{
\mathrm{Fam}X \ar[d]_{\mathrm{forget}}^(0.5){\phantom{aaa}}="1" \ar[rr]^{!}  && \mathbf{1} \ar[d]^{X}_(0.5){\phantom{aaa}}="2" \ar@{~>}"1";"2"^-{\lambda}
\\
\mathrm{Set} \ar[rr]_-{\subset} && \mathrm{Cat} 
}
\end{equation*}
so that functors $f : Y\to \mathrm{Fam}X$ are in 2-natural bijection with pairs $(\tilde{f}, \phi)$
where $\tilde{f} : Y \to \mathrm{Set}$ is a functor and $\phi : \tilde{f} \rightsquigarrow Z!_Y$ is a lax
natural transformation.    
The Grothendieck fibration construction transforms such $(\tilde{f}, \phi)$ into a commutative triangle
\begin{equation*}
\xymatrix{
E \ar[rd]_{q}\ar[rr]^{(\hat{f},q)}   && X\times Y \ar[ld]^{\mathrm{pr}_2} \\
& Y  &
}
\end{equation*}
for which the data are a discrete opfibration $q : E \to Y$ and an arbitrary functor $\hat{f} : E \to X$;
as Lawvere pointed out early in the decade of the 1970s, we might think of this as a 
{\em 2-dimensional partial map} $(q,\hat{f}) : Y\to X$ between categories.
This gives a pseudonatural equivalence of categories
\begin{eqnarray*}
[Y,\mathrm{Fam}X] \simeq 2\mathrm{Par}(Y,X)
\end{eqnarray*}

The product completion of $X$ is $\mathrm{Fam}^{\mathrm{op}}X : = \mathrm{Fam}(X^{\mathrm{op}})^{\mathrm{op}}$. Objects $(I,x)$ of $\mathrm{Fam}^{\mathrm{op}}X$ are
functors $x : I \to X$ from a small discrete category (set) $I$ to $X$, Morphisms
$(u,\xi) : (I,x) \to (J,y)$ are diagrams in $\mathrm{Cat}$ of the form
\begin{equation*}
\xymatrix{
I \ar[rd]_{x}^(0.5){\phantom{a}}="1"  && J  \ar[ll]_{u} \ar[ld]^{y}_(0.5){\phantom{a}}="2" \ar@{=>}"1";"2"^-{\theta}
\\
& X & \ . 
}
\end{equation*}  
Functors $f : Y\to \mathrm{Fam}^{\mathrm{op}}X$ correspond, up to equivalence, to spans
$Y \xla{p} E\xra{g} X$ where $p$ is a discrete fibration; we shall call such a span a
{\em 2-dimensional partial opmap} from $Y$ to $X$.
This gives a pseudonatural equivalence of categories
\begin{eqnarray}
\begin{aligned}
[Y,\mathrm{Fam}^{\mathrm{op}}X] \simeq 2\mathrm{Parop}_{\mathrm{op}}(Y,X) \ .
\end{aligned}
\end{eqnarray}
While there is a size problem with $\mathrm{Fam}^{\mathrm{op}}$ as a monad on $\mathrm{Cat}$,
we do have what would be its Kleisli bicategory, namely,  $2\mathrm{Par}_{\mathrm{op}}$ whose objects are small categories, whose homs are
the categories $2\mathrm{Par}_{\mathrm{op}}(Y,X)$, and whose composition is that of spans. 
There is also a size problem with $\mathrm{Psh}$ as a monad 
$$X \xra{k} Y  \ \mapsto  \ [X^{\mathrm{op}},\mathrm{Set}] \xra{\exists_k = \mathrm{lan}(k^{\mathrm{op}},-)} [Y^{\mathrm{op}},\mathrm{Set}]$$ 
on $\mathrm{Cat}$ but we do have its
Kleisli bicategory $\mathrm{Mod}$ whose objects are small categories, whose homs are given by
$\mathrm{Mod}(Y,X) = [X^{\mathrm{op}}\times Y,\mathrm{Set}]$, 
and composition is that of modules (see Example~\ref{powersinMod}).  
Modulo the size problem, the monad $\mathrm{Psh}$ lifts to a monad $\widehat{\mathrm{Psh}}$
on $2\mathrm{Par}_{\mathrm{op}}$: this is one way of seeing that we have a distributive law
$\partial : \mathrm{Psh}\mathrm{Fam}^{\mathrm{op}} \Lra \mathrm{Fam}^{\mathrm{op}} \mathrm{Psh}$.  
The value of $\widehat{\mathrm{Psh}}$ at a 2-partial opmap $Y \xla{p} E\xra{g} X$ is
\begin{eqnarray}\label{hatPsh}
[Y^{\mathrm{op}},\mathrm{Set}] \xla{\exists_p} [E^{\mathrm{op}},\mathrm{Set}]\xra{\exists_g} [X^{\mathrm{op}},\mathrm{Set}] \ .
\end{eqnarray}
There are a several things to be said about this most of which are better understood by looking at the equivalent span where presheaves are replaced by discrete fibrations:
\begin{eqnarray*}
\mathrm{DFib}Y \xla{p_*} \mathrm{DFib}E \xra{g_*} \mathrm{DFib}X \ .
\end{eqnarray*}
Here $g_* : \mathrm{DFib}E \to \mathrm{DFib}X$ is defined on the discrete fibration $r : F \to E$ by factoring
the composite $g\circ r : F\to X$ as $g\circ r = s \circ j$ where $j : F\to F'$ is final and $s : F' \to X$ is a discrete fibration; this uses the comprehensive factorization of functors described in \cite{6, 104}. In particular, $p_*(r) = p\circ r$ since the composite is already a discrete fibration.
It follows that, if $p : E\to X$ is a discrete fibration then so is $p_* : \mathrm{DFib}E\to \mathrm{DFib}X$.
Also, if further, the left square 
\begin{equation*}
\xymatrix{
F \ar[rr]^-{g} \ar[d]_-{q} && E \ar[d]^-{p} \\
Y \ar[rr]_-{f} && X} 
\qquad
\xymatrix{
\mathrm{DFib}F \ar[rr]^-{g_*} \ar[d]_-{q_*} && \mathrm{DFib}E \ar[d]^-{p_*} \\
\mathrm{DFib}Y \ar[rr]_-{f_*} && \mathrm{DFib}X}
\end{equation*} 
is a pullback, then so is the right square. 
Using this, we conclude that \eqref{hatPsh} is
again a 2-partial opmap and that $\widehat{\mathrm{Psh}}$ is a pseudofunctor. 
To see that the unit for the monad $\mathrm{Psh}$, which is given by Yoneda embedding 
$\mathrm{y}_X : X\to \mathrm{Psh}X$, lifts to $2\mathrm{Par}^{\mathrm{op}}$, we must see that
$\mathrm{y}_X$ seen as a 2-partial opmap, is pseudonatural in $X\in 2\mathrm{Par}_{\mathrm{op}}$;
this follows from the fact that, for all discrete fibrations $p : E\to X$, the square
\begin{equation*}
\xymatrix{
E \ar[rr]^-{E/-} \ar[d]_-{p} && \mathrm{DFib}E \ar[d]^-{p_*} \\
X \ar[rr]_-{X/-} && \mathrm{DFib}X}
\end{equation*}  
is a pullback, which is another form of the Yoneda Lemma.
Rather than examine the multiplication for $\widehat{\mathrm{Psh}}$, as in Example~\ref{toposex}, we take the 
``no iteration'' or ``mw-'' point of view. We need to supply functors
\begin{eqnarray}\label{no-itdatum}
\mathrm{P} :  2\mathrm{Par}_{\mathrm{op}}(C,\mathrm{Psh}X) \lra 2\mathrm{Par}_{\mathrm{op}}(\mathrm{Psh}C,\mathrm{Psh}X) \ .
\end{eqnarray}
An object of the domain is a span $C \xla{p} Z\xra{g} \mathrm{Psh}X$ where $p$ is a discrete fibration.
Define 
$$\mathrm{P}(C \xla{p} Z\xra{g}\mathrm{Psh}X) = (\mathrm{Psh}C \xla{p_*} \mathrm{Psh}Z\xra{\bar{g}}\mathrm{Psh}X)$$
where $\bar{g} = \mathrm{lan}(\mathrm{y}_Z, g)$ is the colimit-preserving extension of $g$.
Thus the composite of $D \xla{q} B\xra{m} \mathrm{Psh}C$ and $C \xla{p} Z\xra{g} \mathrm{Psh}X$ in the
Kleisli bicategory $(2\mathrm{Par}_{\mathrm{op}})_{\widehat{\mathrm{Psh}}}$ of the pseudomonad
$\widehat{\mathrm{Psh}}$ is the following composite of spans in $\mathrm{Cat}$.    
\begin{eqnarray*}
\xymatrix{
& & Y\ar[rd]^-{n} \ar[ld]_-{r} & & \\
& B \ar[rd]^-{m} \ar[ld]_-{q} &  & \mathrm{Psh}Z \ar[rd]^-{\bar{g}} \ar[ld]_-{\exists_p}& \\
D & & \mathrm{Psh}C & & \mathrm{Psh}X }
\end{eqnarray*}
Suppose the left square in diagram \eqref{bothpbs} is in $\mathrm{Mod}$ and the right is in $\mathrm{CAT}$.  
\begin{eqnarray}\label{bothpbs}
\begin{aligned}
\xymatrix{
K \ar[d]_{h}^(0.5){\phantom{aaaaa}}="1" \ar[rr]^{t_*}  && B \ar[d]^{m}_(0.5){\phantom{aaaaa}}="2" \ar@{=>}"1";"2"_-{\cong}^-{\theta}
\\
Z \ar[rr]_-{p_*} && C 
}
\qquad
\xymatrix{
K \ar[d]_{h}^(0.5){\phantom{aaaaaaa}}="1" \ar[rr]^{t}  && B \ar[d]^{m}_(0.5){\phantom{aaaaaaa}}="2" \ar@{=>}"1";"2"_-{\cong}^-{\phi}
\\
\mathrm{Psh}Z \ar[rr]_-{\exists_p} && \mathrm{Psh}C 
}
\end{aligned}
\end{eqnarray}
An easy evaluation shows that isomorphisms $\theta$ are in bijection with isomorphisms $\phi$. 
It follows that $Y$, $r$ and $n$ agree with the construction in \eqref{construction} and we
have the biequivalence
\begin{eqnarray*}
\mathrm{PolyMod}^{\mathrm{op}} \simeq (2\mathrm{Par}_{\mathrm{op}})_{\widehat{\mathrm{Psh}}}
\end{eqnarray*}
from which we obtain the claim of this example's first paragraph.

Incidentally, using this biequivalence, we can view the pseudofunctor $\mathbb{H}_K$
of Proposition~\ref{bbH} as the pseudofunctor
$$(2\mathrm{Par}_{\mathrm{op}})_{\widehat{\mathrm{Psh}}}^{\mathrm{op}}\lra \mathrm{Cat}$$
taking the morphism $Y\xla{p}S\xra{m}\mathrm{Psh}$ to the functor
\begin{eqnarray}
[K,\mathrm{Psh}X]\lra [K,\mathrm{Psh}Y] \ , \ \ell \mapsto \bar{\ell}
\end{eqnarray}
where $$(\bar{\ell}k)y = \sum_{s\in S_y}{\mathrm{Psh}X(ms,\ell k)}$$ 
for $k\in K$, for $y\in Y$ and for $S_y$ the fibre of $p : S\to Y$ over $y$.
\end{example}
\begin{center}
--------------------------------------------------------
\end{center}

\appendix

\end{document}